\documentclass[12pt]{amsart}
\usepackage{bm,amsmath,amsthm,amssymb,mathtools,verbatim,amsfonts,tikz-cd,mathrsfs,mathabx}
\usepackage{enumitem}
\usepackage{float}
\usepackage{mathtools}
\usepackage[all,cmtip]{xy}
\usepackage{stmaryrd}
\usepackage[hidelinks,colorlinks=true,linkcolor=blue, citecolor=black,linktocpage=true]{hyperref}
\usepackage{graphicx}
\usepackage[alphabetic]{amsrefs}
\usepackage{caption}

\usepackage{listings}
\usepackage{xcolor} 

\definecolor{codegreen}{rgb}{0,0.6,0}
\definecolor{codegray}{rgb}{0.5,0.5,0.5}
\definecolor{codepurple}{rgb}{0.58,0,0.82}
\definecolor{backcolour}{rgb}{0.95,0.95,0.92}

\lstdefinestyle{mystyle}{
    commentstyle=\color{codegreen},
    keywordstyle=\color{magenta},
    numberstyle=\tiny\color{codegray},
    stringstyle=\color{codepurple},
    basicstyle=\ttfamily\footnotesize,
    breakatwhitespace=false,         
    breaklines=true,                 
    captionpos=b,                    
    keepspaces=true,                 
    numbers=left,                    
    numbersep=5pt,                  
    showspaces=false,                
    showstringspaces=false,
    showtabs=false,                  
    tabsize=2,
    frame=single, % Adds a frame around the code
    rulecolor=\color{black} % Color of the frame
}
\usepackage[top=1.3in, bottom=1.1in, left=1.3in, right=1.3in,marginpar=1in]{geometry}
\usepackage{subfigure}
\usepackage{pgfplots}
\pgfplotsset{compat=1.18}
\counterwithin{figure}{section}
\newtheorem{theorem}{Theorem}[section]
\newtheorem{prop}[theorem]{Proposition}
\newtheorem{conjecture}[theorem]{Conjecture}
\newtheorem{corollary}[theorem]{Corollary}
\newtheorem{lemma}[theorem]{Lemma}

\theoremstyle{definition}
\newtheorem{defn}[theorem]{Definition}

\theoremstyle{remark}
\newtheorem{remark}[theorem]{Remark}
\newtheorem{example}[theorem]{Example}

\newtheorem{question}[theorem]{Question}

\DeclareMathOperator{\coker}{coker}
\DeclareMathOperator{\gr}{gr}
\DeclareMathOperator{\ad}{ad}

\newcommand\SU{\mathrm{SU}}

\numberwithin{equation}{section}
\DeclareMathOperator{\nll}{null}
\DeclareMathOperator{\rank}{rank}
\DeclareMathOperator{\im}{im}

\title{Framed instanton homology and Fr{\o}yshov's invariant}

\author{Sudipta Ghosh}

\address{University of Notre Dame\\ USA}
\email{sghosh7@nd.edu}

\author{Mike Miller Eismeier}

\address{University of Vermont\\ USA}
\email{Mike.Miller-Eismeier@uvm.edu}
\begin{document}
\maketitle

\begin{abstract}
We determine the framed instanton homology with coefficients in $\mathbb F = \mathbb Z/2$ for Dehn surgeries on a knot in the $3$-sphere. The dimension of these groups is seen to have a close relationship with a homology cobordism invariant due to Fr{\o}yshov.

As an application, we show that $r$-surgery on a non-trivial knot cannot be nondegenerate $\SU(2)$-abelian for any $|r| \le 4\lceil g(K)/2\rceil$, which is $2g(K)$ for $g$ even and $2g(K) + 2$ for $g$ odd.
\end{abstract}

\section{Introduction}
The framed instanton homology, introduced by Kronheimer and Mrowka in \cite{KMframed}, associates to any closed, oriented $3$-manifold $Y$ a $\mathbb Z/2$-graded abelian group $I^\#(Y)$, defined in terms of Floer's instanton homology on the $3$-manifold $Y \# T^3$. In simple cases, $I^\#(Y)$ is the target of a spectral sequence beginning with the homology of the framed representation variety $\text{Hom}(\pi_1 Y, \SU(2))$, and so can be used to study representations of the fundamental group.

Many different-looking variations of Floer homology for $3$-manifolds have been introduced, notably Heegaard Floer homology \cite{OSdisks}, embedded contact homology \cite{hutchings-index}, and monopole Floer homology \cite{KM-book}. The three named Floer homologies have very different definitions, yet have been shown to be isomorphic \cite{CGHhf=ech1, klt1}. The relationship these theories share with Floer's instanton homology remains unknown. Kronheimer and Mrowka proposed the following:

\begin{conjecture}\cite{km-excision}*{Conjecture~7.24}\label{Conj: KM}
    When taken with complex coefficients, there exists an isomorphism $I^\#(Y;\mathbb C) \cong \widehat{HF}(Y;\mathbb C)$. 
\end{conjecture}

This conjecture has been verified for a large class of 3-manifolds, such as the boundaries of almost-rational plumbing \cite{BSLY-plumbing}, Seifert fibered rational homology spheres \cite{LFapplications}, integral surgeries of L-space knots \cite{LCS-lspace}, and non-zero surgeries of alternating knots of bridge index at most 3 \cite{GL-alternating}. 

Partly towards understanding Conjecture \ref{Conj: KM}, in \cite{bs-concordance} Baldwin and Sivek compute both the framed instanton homology and Heegaard Floer homology of Dehn surgeries on a large class of knots $K \subset S^3$. Their computations stem from the following general result.

\begin{theorem}[\cite{bs-concordance}*{Theorem~1.1}]\label{thm:rational-dimformula}
Let $K \subset S^3$ be a knot. There exist integers $r_0(K), \nu^\#(K)$ so that, if $p/q$ is written in least terms with $q \ge 0$, then \[\dim I^\#(S^3_{p/q}(K);\mathbb C) = \begin{cases} qr_0(K) + |p - q\nu^\#(K)| & \nu^\#(K) \ne 0 \text{ or } p/q \ne 0 \\ r_0(K) \text{ or } r_0(K) + 2 & p/q = \nu^\#(K) = 0. \end{cases}\] 
\end{theorem}

Focusing only on integer surgeries, this theorem asserts that the graph of the function $\dim I^\#(S^3_n(K);\mathbb C)$ is `$V$-shaped' except possibly when $\nu^\#(K) = 0$, when it is either `$V$-shaped' or `$W$-shaped'. When $\nu^\#(K) = 0$, the dimension of $I^\#(S^3_0(K);\mathbb C)$ is not known to be determined by $r_0(K)$ alone. 

More generally, one may define the framed instanton homology $I^\#(Y,w)$ for any pair of a closed oriented $3$-manifold $Y$ and an oriented $1$-manifold $w \subset Y$;  the resulting homology groups depend up to isomorphism only on the homology class $[w]_2 \in H_1(Y;\mathbb Z/2)$. Baldwin and Sivek show that $\dim I^\#(S^3_{p/q}(K),w;\mathbb C)$ is independent of the choice of $w$ except in the case $\nu^\#(K) = 0$.

While $\widehat{HF}(Y)$ is typically torsion-free and thus the dimension of $\widehat{HF}(Y;k)$ is independent of the field $k$, it seems to be typical that $I^\#(Y)$ contains $2$-torsion. This is already the case for the Poincar\'e sphere \cite{bhat}*{Theorem~1.5}, and the main theorem of \cite{LiYe1} shows that $I^\#(S^3_r(K);\mathbb Z)$ contains $2$-torsion for $r \in \{1, 1/2, 1/4\}$ for any nontrivial knot $K$. Thus, Conjecture \ref{Conj: KM} fails over $\mathbb Z$ and $\mathbb F = \mathbb Z/2$.

Indeed, the behavior of $I^\#$ is significantly different when taken with $\mathbb F$ coefficients. The eigenvalue decomposition of \cite{bs-lspace} is not available in characteristic $2$, there is no known analogue of the adjunction inequality, and the landscape is less clear. 

Nevertheless, Li and Ye prove that the graph of dimensions of the homology groups, $\dim I^\#(S^3_n(K); \mathbb{F})$ is always $V$-shaped, $W$-shaped, or a `generalized $W$-shape' \cite{LiYe1}*{Proposition~1.11}. Our first main result is an analogue of Theorem \ref{thm:rational-dimformula}, with a crucial distinction: we show that the graph is always `$W$-shaped'.

\begin{theorem}\label{thm:surgerydim}
Let $K \subset S^3$ be a knot. There exist integers $M(K)$ and $r_2(K)$, both divisible by $4$ and with $r_2(K) \ge |M(K)|$, with the following significance: for any rational $p/q$ in least terms with $q \ge 0$, and any $1$-cycle $w$, we have \begin{equation}\label{eqn:dim-formula}\dim I^\#(S^3_{p/q}(K), w; \mathbb F) = \begin{cases}  r_2(K) + 2 & p/q = M(K) \text{ and } [w]_2 = 0 \\ 
qr_2(K) + |p - qM(K)| & \text{otherwise.} \end{cases}\end{equation}
\end{theorem}

With the exception of slope $p/q = M(K)$, this formula has also been established by Li and Ye by different techniques \cite{LiYe3}.

As a consequence of \eqref{eqn:dim-formula}, for every knot $K$ in the $3$-sphere there is exactly one slope for which $\dim I^\#(S^3_r(K);\mathbb F) \ne \dim I^\#(S^3_r(K),w;\mathbb F)$ depends on the choice of $w$, in which case the former group has larger dimension. This can be taken as a definition of $M(K)$, and we can define $r_2(K) = \dim I^\#(S^3_{M(K)}, w; \mathbb F)$ for $[w]_2 \ne 0$. 

Baldwin and Sivek give a relationship between $\nu^\#(K)$ and Fr{\o}yshov's $h$-invariant $h(S^3_1(K))$ \cite{bs-concordance2}*{Proposition~9.2}. Fr{\o}yshov has also introduced integer homology cobordism invariants $q_2(Y), q_3(Y) \in \mathbb Z$ \cite{froy-mod2}. The invariant $q_3(Y)$ extends to an integer-valued $\mathbb F$-homology cobordism invariant of $\mathbb F$-homology spheres, and the behavior of $q_3(Y)$ under Dehn surgery characterizes the invariant $M(K)$.

\begin{theorem}\label{thm:q3-surgery}
Fr\o yshov's invariant $q_3$ satisfies, for all $r \in \mathbb Q$ with odd numerator, \[q_3(S^3_r(K)) = \begin{cases} 1 & M(K) < r < 0 \\ -1 & 0 < r < M(K) \\ 0 & \text{otherwise.} \end{cases}\] 
\end{theorem}

As a consequence, like $\nu^\#(K)$, the integer $M(K)$ is a concordance invariant. Unlike $\nu^\#(K)$, it is even and additive under connected sum.

\begin{theorem}\label{thm:M-additive}
$M(K)$ is a concordance invariant and defines a surjective homomorphism $M: \mathcal C \to 4\mathbb Z$. 
\end{theorem}

A pair $(Y,w)$ is said to be an \textbf{instanton $L$-space over $\mathbb F$} if $\dim I^\#(Y,w;\mathbb F) = |H_1(Y;\mathbb Z)|$. This is the minimal possible dimension, as the Euler characteristic of $I^\#(Y,w)$ is equal to $|H_1(Y;\mathbb Z)|$. A knot $K \subset S^3$ is said to be a (positive) \textbf{instanton $L$-space knot over $\mathbb F$} if some surgery $S^3_t(K)$ (with $t > 0$) is an instanton $L$-space over $\mathbb F$. By the universal coefficient theorem, these are also $L$-space knots over $\mathbb Q$. 

A typical example is a \textit{nondegenerate $\SU(2)$-abelian $3$-manifold}, for which every representation $\pi_1(Y) \to \SU(2)$ has abelian image and satisfies a certain nondegeneracy condition; see Section \ref{sec:TAK} for more details. These include, but are by no means limited to, lens spaces, and the nondegeneracy condition is automatic when $Y = S^3_{p/q}(K)$ for $p$ a prime power.

\begin{corollary}\label{cor:F2Lspace}
A knot $K \subset S^3$ is an instanton $L$-space knot over $\mathbb F$ if and only if $r_2(K) = |M(K)|$. If $K$ is a positive instanton $L$-space knot over $\mathbb F$, then $(S^3_r(K),w)$ is an instanton $L$-space over $\mathbb F$ if and only if $r > M(K)$ or $r = M(K)$ and $[w]_2 \ne 0$. 
\end{corollary} 
\begin{proof}
If $r_2(K) = |M(K)|$ it follows from \eqref{eqn:dim-formula} that $K$ is an instanton $L$-space knot over $\mathbb F$. Conversely, suppose $S^3_{p/q}(K)$ is an instanton $L$-space; we cannot have $p/q = M$ because in this case \[\dim I^\#(S^3_M(K);\mathbb F) = \dim I^\#(S^3_M(K),w;\mathbb F) + 2 \ge |M| + 2 \ne |M|.\] Therefore, \eqref{eqn:dim-formula} gives $qr_2(K) + |p-qM(K)| = |p|.$ Thus \[qr_2(K) = |p| - |p-qM(K)| \le |qM(K)| = q|M(K)|,\] we have $r_2(K) \le |M(K)|$, giving equality.
\end{proof}

The main result of \cite{LiYe2} establishes that $I^\#(S^3_n(K);\mathbb F)$ has no $2$-torsion for any $|n| < 2g(K) - 1 + t_2(K)$, where $2t_2(K)$ is the quantity \[\dim I^\#(S^3_1(K);\mathbb F) - \dim I^\#(S^3_1(K);\mathbb Q).\]
In Section \ref{sec:TAK} we show their result is essentially sharp, by showing that if $S^3_n(K)$ ever admits $2$-torsion then $|M(K)| = 2g(K) - 1 + t_2(K)$. 

At this point, we compare our invariants to Baldwin and Sivek's more directly. As a consequence of the universal coefficient theorem, we have \[r_2(K) \ge r_0(K), \quad r_2(K) - r_0(K) \ge |M(K) - \nu^\#(K)|.\] 

A computation similar to Corollary \ref{cor:F2Lspace} shows that we have equality in this latter case if and only if $I^\#(S^3_r(K);\mathbb Z)$ is free of $2$-torsion for some rational number $r$, in which case we call $K$ `torsion-averse'. 

It is established in \cite{LiYe2}*{Theorem~1.4} that a torsion-averse knot $K$ is in fact an instanton $L$-space knot over $\mathbb F$. Thus $r_2(K) = |M(K)|$ and $r_0(K) = |\nu^\#(K)|$. The inequality \[|M(K)| - |\nu^\#(K)| \ge |M(K) - \nu^\#(K)|\] then implies $M(K), \nu^\#(K)$ have the same sign and $0 < |\nu^\#(K)| \le |M(K)|$. 

By \cite{bs-concordance}*{Theorem~1.18}, we have $|\nu^\#(K)| = 2g(K)-1$. Because $M$ is divisible by $4$, it is bounded below by \[|M(K)| \ge 4\bigg\lfloor{\frac{2g(K)-1}{4}}\bigg\rfloor = 4\lfloor g(K)/2 \rfloor = \begin{cases} 2g(K) + 2 & g(K) \text{ odd} \\ 2g(K) & g(K) \text{ even} \end{cases}\]

which gives us most of the following result.

\begin{corollary}\label{cor:SU2-ab}
The manifold $S^3_r(K)$ is never nondegenerate $\SU(2)$-abelian for any $|r| \le 4\lfloor g(K)/2 \rfloor$. If $S^3_t(K)$ is nondegenerate $SU(2)$-abelian for some $t > 0$, then $S^3_r(K)$ is not even $\mathbb F$-homology cobordant to a nondegenerate $SU(2)$-abelian $3$-manifold for any $0 < r < 4\lfloor g(K)/2 \rfloor$ with odd numerator.
\end{corollary}
\begin{proof}
If $S^3_r(K)$ is nondegenerate $\SU(2)$-abelian, then $K$ is an $L$-space knot over every field and in particular over $\mathbb F$. It follows that $|r| > |M(K)| \ge 4\lfloor g(K)/2 \rfloor$. For $0 < r < M(K)$ with odd numerator, we have $q_3(S^3_r(K)) = -1 \ne 0$. We prove in Lemma \ref{lemma:ndga-q3} that for a nondegenerate $\SU(2)$-abelian $3$-manifold, we have $q_3(Y) = 0$. It follows that $S^3_r(K)$ is not $\mathbb F$-homology cobordant to any nondegenerate $\SU(2)$-abelian $3$-manifold.
\end{proof}

As a special case, we have the following strengthening of \cite{LiYe2}*{Corollary~1.18.}
\begin{corollary}
    If $K \subset S^3$ is a knot for which $S^3_7(K)$ is $\SU(2)$-abelian, then $K$ is the unknot or right-handed trefoil.
\end{corollary}
\begin{proof}
    Li--Ye prove that $K$ is the unknot, right-handed trefoil, or an $L$-space knot over $\mathbb F$ with $g(K) = 3$. However, in this case we have $|M(K)| \ge 8$, so $I^\#(S^3_7(K);\mathbb Z)$ contains $2$-torsion, and is therefore not $\SU(2)$-abelian.
\end{proof}

To compute the invariants $r_2(K)$ and $M(K)$ for a given knot, one needs to know the dimension of $I^\#(S^3_r(K);\mathbb F)$ for $|r|$ sufficiently large, and to know the value of $M(K)$. This quantity seems to be difficult to compute, but the following is helpful.

\begin{theorem}\label{thm:crossing-change}
Suppose $K$ can be modified to $K'$ by changing $p$ positive and $n$ negative crossings of $K$. Then \[-4n \le M(K) - M(K') \le 4p.\]
\end{theorem}

Because $M(K)$ is a concordance invariant, we have $-4c_+(K) \le M(K) \le 4c_-(K)$, where $c_\pm$ are the signed \textit{clasp numbers}, representing the minimal number of positive (resp. negative) double points in an immersed disc in $B^4$ bounding $K$ \cite{OwensStrle-immersed}. 

We can use these relationships to compute the invariants for the two infinite families below. This gives, in particular, an infinite family of knots for which $r_2(K) + |M(K)| = 16$, and an infinite family for which $r_2(K)$ is arbitrarily large. 

\begin{theorem}\label{thm:twist-comp}
Take $n \ge 1$. For the odd twist knots $K_{2n-1}$, we have $r_2(K_{2n-1}) = 8n-4$ and $M(K_{2n-1}) = -4$. For the even twist knots $K_{2n}$, we have $r_2(K_{2n}) = 8n$ and $M(K_{2n}) = 0$. For the pretzel knots $P_n = P(n, -3, 3)$, we have $r_2(P_n) = 16$ and $M(P_n) = 0$.
\end{theorem}

In these examples, we have $r_2(K) = 4r_0(K)$ and $M(K) = 4\nu^\#(K)$. This cannot hold in general. First, $M$ is additive while $\nu^\#$ is known not to be. Secondly, for any torus knot $T_{p,q}$ with $0 < p < q$ has $S^3_{-pq-1}(T_{p,q})$ a lens space. We have 
\[|\nu^\#(K)| = pq-p-q, \quad pq-p-q \le |M(K)| \le pq-2.\] 
No multiple of $M$ is a slice-torus invariant.

\subsection*{Questions} 
This limited data suggests further lines of inquiry.

\begin{question}
    Is it the case that $r_2(K) \equiv M(K) \mod 8$? 
\end{question}

The second condition would follow if it were known that $\dim I^\#(Y,w;\mathbb F)$ were divisible by $8$ for any admissible bundle. 

\begin{question}
Is it the case that $|M(K)| \le 4g_4(K)$ is bounded by a multiple of the slice genus?
\end{question}

This is true for twist knots and torus knots.  Even if this bound were proved, we have a wide range between the bounds $|M(K)| \in [2g_4(K), 4g_4(K)]$. 

\begin{question}
What pairs of values $(g_4(K), M(K))$ are realized for $K$ an instanton $L$-space knot over $\mathbb F$?
\end{question}

It is a consequence of \cite{LiYe1}*{Theorem~1.1} that if $K$ is a non-trivial knot, we have $r_2(K) - r_0(K) > 0$. The gap is larger in all known examples, and when $K$ is torsion-averse this quantity determines the surgeries for which $I^\#(S^3_r(K);\mathbb Z)$ has no $2$-torsion. It would be interesting to better determine it, especially in cases where it is small.

\begin{question}
Is it possible to enumerate knots with, say, $r_2(K) \le 12$? What about $r_2(K) - r_0(K) \le 9$?
\end{question}

\subsection*{Dependence on forthcoming work} The results of this paper depend on the forthcoming work \cite{DMES,DME2}. The second author wishes to apologize for the long delay since their announcement. Section 2 below gives all necessary background from those references, and it is intended that this paper is otherwise self-contained. Nevertheless, we briefly summarize the dependence on these works. 

From \cite[Chapter 8]{DMES} we use the definition of the equivariant instanton complexes, which are finite-dimensional dg-modules $\widetilde C(Y,w,\mathfrak a)$ associated to a rational homology sphere with some auxiliary data, and their functoriality properties. We also use a result relating $\widetilde C$ to $I^\#$, both at the level of $3$-manifolds and at the level of cobordism maps. Finally, we use explicit computations of $I^\#(Y;\mathbb Z)$ for $Y$ the manifolds $\Sigma(2,3,6 \pm 1)$ and for the surgeries $S^3_n(T_{2,3})$ for $-5 \le n \le -1$.

From \cite{DME2} we use two things. The main result gives an inequality on the difference $q_3(Y') - q_3(Y)$ where $\partial W = Y' - Y$, and we take this as a black box. That reference also constructs chain maps associated with cobordisms with $b^+(W) = 1$, and we need to know these may be compared to the induced map on $I^\#(Y)$.

\subsection*{Organization} 
In Section \ref{sec:bg} we summarize the material mentioned above which will be used in the proof, as well as the relevant exact triangles. In Section \ref{sec:integer-case}, which is the technical heart of this paper, we prove Theorem \ref{thm:surgerydim} and Theorem \ref{thm:q3-surgery} for integer surgeries, proving Theorem \ref{thm:M-additive} along the way. We extend these results to all rational surgeries in Section \ref{sec:rat}. In Section \ref{sec:TAK} we discuss the implications of our dimension formula for instanton $L$-spaces, completing the proof of Corollary \ref{cor:SU2-ab}. We conclude in Section \ref{sec:comp} by proving Theorem \ref{thm:crossing-change} and Theorem \ref{thm:twist-comp}.

\subsection*{Acknowledgements} 
The authors thank John Baldwin, Christopher Scaduto, Tye Lidman, and Steven Sivek for helpful remarks during the preparation of this work. The authors are especially grateful to Zhenkun Li and Fan Ye for pointing out an important subtlety in the proof of Theorem \ref{thm:surgerydim} in a previous draft of this article.

\section{Background}\label{sec:bg}
In Sections \ref{subsec:DMES1}-\ref{subsec:DMES3}, we state the results of \cite{DMES} that are used in this paper. In Section \ref{subsec:DME2} we do the same for \cite{DME2}. In Section \ref{subsec:SET} we recall the behavior of the surgery exact triangles that will be used in this paper.

\subsection{The complex $\widetilde C$}\label{subsec:DMES1}
We first review the construction of the tilde complex. Most material in this section can be found in \cite[Chapter 8.1-8.2]{DMES}; the explicit description of the differential and chain map is given in \cite[Chapter 7.2]{DMES}.

A pair $(Y,w)$ of a closed oriented $3$-manifold and oriented $1$-manifold $w \subset Y$ is called \emph{admissible} if there exists an oriented surface $\Sigma \subset Y$ with $\Sigma \cdot w \equiv 1 \mod 2$, and \emph{weakly admissible} if $(Y,w)$ is admissible or $b_1(Y) = 0$.

The oriented $1$-manifold $w$ gives rise to a $U(2)$-bundle $E_w \to Y$ with $c_1(E_w) = \text{PD}([w])$. We will study the space $\widetilde B_{(Y,w)}$ of $U(2)$-connections on $E_w$ with trace equal to a specified connection $A_0$, considered modulo determinant-$1$ gauge transformations which are equal to the identity at a chosen basepoint $y \in Y \setminus w$. This space comes equipped with an $SO(3)$-action, and a functional $\text{CS}: \widetilde B_{(Y,w)} \to \mathbb R/\mathbb Z$. The invariants discussed here are derived from the equivariant Morse theory of this Chern--Simons functional.

Suppose now that $(Y,w)$ is weakly admissible. If $(Y,w)$ is admissible, then the critical set of $\text{CS}$ consists of free $SO(3)$-orbits, called \textit{irreducible}. When $Y$ is a rational homology sphere, the critical set also contains components diffeomorphic to $S^2$ --- called \textit{abelian} --- and singleton components, called \textit{central}. The set of abelian and central connections can be determined cohomologically \cite{DME1}*{Lemma~2.1}:
\begin{align*}
\mathfrak A(Y, w) &= \big\{\{z_1, z_2\} \subset H^2(Y;\mathbb Z) \mid z_1 + z_2 = \textup{PD}(w), \;\; z_1 \ne z_2\} \\
\mathfrak Z(Y, w) &= \big\{z \in H^2(Y;\mathbb Z) \mid 2z = \text{PD}(w)\big\}.
\end{align*}

Notice in particular that $\mathfrak Z(Y,w)$ is empty unless $[w]_2 = 0 \in H_1(Y;\mathbb F)$, and when $[w] = 0 \in H_1(Y;\mathbb Z)$, there is a canonical element $\theta \in \mathfrak Z(Y,w)$ corresponding to the cohomology class $z = 0$. Geometrically, when $w = \varnothing$, this should be understood as the trivial connection on the trivial bundle $E_w$. 

For each $\alpha \in \mathfrak A(Y,w)$, we denote by $H^1(Y; \mathbb C_{\textup{ad } \alpha})$ the cohomology with coefficients in the adjoint representation.

\begin{defn}
If $(Y,w)$ is a rational homology sphere  \textbf{abelian data} $\mathfrak a = (\sigma, s)$ on $(Y,w)$ is the choice of two functions: 
\begin{itemize}
    \item $\sigma: \mathfrak A(Y,w) \to 2\mathbb Z$ is a function with $\sigma(\alpha) \equiv \dim_{\mathbb R} H^1(Y; \mathbb C_{\textup{ad } \rho}) \pmod{4}.$
    \item $s: \mathfrak A(Y,w) \to H^2(Y;\mathbb Z)$ is a function with $s\{z_1, z_2\} \in \{z_1, z_2\}$.
\end{itemize}
\end{defn}

When studying the Morse theory of the Chern--Simons functional, we must perturb it. The first piece of data corresponds to a choice of `chamber' in the space of perturbations of the Chern--Simons functional on $(Y,w)$. The second piece of data corresponds to choosing an $SO(2)$-fixed basepoint on each abelian orbit, which is used in the definition of the cellular Morse--Bott complex.

There is a dg-module over $C_*^{\text{cell}}(SO(3);\mathbb F) \cong \Lambda_{\mathbb F}(\chi_1, \chi_2)$, denoted $\widetilde C(Y, w, \mathfrak a; \mathbb F)$, which depends on a choice of abelian data $\mathfrak a$. The underlying vector space of $\widetilde C$ is equipped with a decomposition 
\begin{equation}\label{eqn:CAZ}\widetilde C = C \oplus \chi_1 C \oplus \chi_2 C \oplus \chi_3 C \oplus A \oplus \chi_2 A \oplus Z,\end{equation}
where each of $C(Y,w), A(Y,w), Z(Y,w)$ are $\mathbb F$-vector spaces, the latter two equipped with bases in bijection with $\mathfrak A(Y,w)$ and $\mathfrak Z(Y, w)$, respectively. In particular, $Z = 0$ if $[w]_2 \ne 0$, and when $w = \varnothing$ there is a canonical basis vector $\theta \in Z$. 

With respect to the direct sum decomposition, the matrix form of the differential is given by
\[\widetilde d = \left[\begin{array}{cccc|cc|c}
d_1 & 0 & 0 & 0 & 0 & 0 & 0 \\
d_2 & d_1 & 0 & 0 & e_2 & 0 & 0 \\
d_3 & 0 & d_1 & 0 & 0 & 0 & 0 \\
d_4 & d_3 & d_2 & d_1 & e_4 & e_2 & \delta_4 \\
\hline
e_1 & 0 & 0 & 0 & 0 & 0 & 0 \\
e_3 & 0 & e_1 & 0 & 0 & 0 & 0 \\
\hline
\delta_1 & 0 & 0 & 0 & 0 & 0 & 0 
\end{array}\right].\]

Suppose $(W,c): (Y,w, \mathfrak a) \to (Y', w', \mathfrak a')$ is a cobordism between weakly admissible pairs equipped with auxiliary data, where $c \subset W$ is an embedded surface with boundary $w' - w$. The notion of `nice cobordism' is introduced in \cite[Chapter 8.2]{DMES}. For a nice cobordism, there exists an induced $\Lambda_{\mathbb F}(\chi_1,\chi_2)$-equivariant chain map
\[(W,c)_*: \widetilde C(Y, w, \mathfrak a; \mathbb F) \to \widetilde C(Y', w', \mathfrak a'; \mathbb F),\] which is well-defined up to equivariant homotopy. Composites of nice cobordisms are nice, and the induced maps compose functorially up to equivariant homotopy.

The definition of `nice cobordism' is somewhat complicated and we will not reproduce it. It will suffice to know that the following cobordisms are nice.

\begin{itemize}
\item $I \times Y: (Y, w, \mathfrak a) \to (Y, w, \mathfrak a')$ is nice if and only if for all $\alpha \in \mathfrak A(Y,w)$ we have $\sigma(\alpha) \le \sigma'(\alpha)$ with strict inequality if $s(\alpha) \ne s'(\alpha)$. 
\item If $(W,c): (Y,w) \to (Y', w')$ is a cobordism between rational homology spheres and $b^+(W) = 0$ or $[c]_2 \ne 0 \in H^2(W;\mathbb F)$, there exists choices of abelian data $\mathfrak a, \mathfrak a'$ so that $(W,c): (Y, w, \mathfrak a) \to (Y', w', \mathfrak a')$ is nice.
\end{itemize}

Suppose $(W,c)$ is nice. With respect to the direct sum decomposition \eqref{eqn:CAZ}, the matrix form of this chain map is given by 
\[\widetilde I(W,c;\mathbb F) = \left[\begin{array}{cccc|cc|c} 
\lambda_0 & 0 & 0 & 0 & 0 & 0 & 0 \\
\lambda_1 & \lambda_0 & 0 & 0 & \mu_1 & 0 & 0 \\
\lambda_2 & 0 & \lambda_0 & 0 & 0 & 0 & 0 \\
\lambda_3 & \lambda_2 & \lambda_1 & \lambda_0 & \mu_3 & \mu_1 & \Delta_3 \\
\hline
\mu_0 & 0 & 0 & 0 & \nu_0 & 0 & 0 \\
\mu_2 & 0 & \mu_0 & 0 & \nu_2 & \nu_0 & \xi_2 \\
\hline
\Delta_0 & 0 & 0 & 0 & \xi_0 & 0 & \epsilon_0 
\end{array}\right].\]

The component $\epsilon_0: Z \to Z'$ will be of particular interest. We summarize the discussion concluding \cite[Chapter 8.2]{DMES}. Suppose $Y, Y'$ are rational homology spheres and $w = w' = \varnothing$. If $[c] \ne 0 \in H^2(W;\mathbb F)$, then $\epsilon_0$ is identically zero. \\

Suppose instead that $[c]_2 = 0$, and write \[T(W) = \# \{x \in \text{Tors } H^2(W;\mathbb Z) \mid i_Y^*(x) = 0 = i_{Y'}^*(x)\}\] for the number of torsion classes which restrict trivially to the boundary. Then 
\begin{equation}\label{eqn:eps0}
    \langle \epsilon_0 \alpha, \alpha'\rangle = \begin{cases} T(W) & \exists x \in \text{Tors } H^2(W;\mathbb Z) \text{ with } i_Y^*(x) = \alpha, \;\; i_Y'^*(x) = \alpha' \\ 
0 & \text{otherwise}
\end{cases}
\end{equation} where here we identify elements of $\mathfrak Z(Y,\varnothing)$ with $2$-torsion cohomology classes. 

There are two special cases worth emphasizing. First, suppose that $H_1(W;\mathbb F) = 0$. Then 
\[\langle \epsilon_0 \alpha, \alpha'\rangle = \begin{cases} 1 & \alpha = \theta, \alpha' = \theta' \\ 0 & \text{otherwise}.\end{cases}\]

This is because $\text{Tors } H^2(W;\mathbb Z)$ is a finite group of odd order, so zero is the only $2$-torsion class in the image of the restriction map to $H^2(\partial W;\mathbb Z)$ and its kernel is finite of odd order.

Second, suppose both inclusion maps induce isomorphisms \[H_1(Y;\mathbb F) \cong H_1(W;\mathbb F) \cong H_1(Y';\mathbb F).\] Then the map $\epsilon_0$ is an isomorphism. When $W = I \times Y$, the map $\epsilon_0$ is the identity.

\begin{remark}
Let $(W,c): (Y,w) \to (Y', w')$ be a cobordism, Define $\epsilon_0(W,c): Z(Y,w) \to Z(Y',w')$ to be zero when $[c]_2 \ne 0$, and to be defined by \eqref{eqn:eps0} otherwise. When $b^+(W) = 0$, we have proved this is one of the components of $\widetilde I(W,c;\mathbb F)$, but we will use the expression $\epsilon_0$ even when $b^+(W) > 0$. 
\end{remark}

\subsection{Equivariant homology and $q_3(Y)$}\label{subsec:DMES2}
We first review the definition of $q_3(Y)$ from \cite[Chapter 11.4]{DMES}. While that reference only presents $q_3(Y)$ in the case that $Y$ is an integer homology sphere, we explain below that the definition may be generalized. This more general definition will appear in the forthcoming work \cite{DME2}, as will some properties used in this paper. 

The general observation is that when $Y$ is an $\mathbb F$-homology sphere, $\widetilde C(Y)$ is a $\mathcal S$-complex in the sense of \cite{DS1}*{Section~4} with respect to the action of $\chi_2$, so we may mimic the definition of the $h$-invariant of a $\mathcal S$-complex. It will serve us to discuss the algebra somewhat generally.

Given a dg-module $\widetilde C = \widetilde C(Y, w, \mathfrak a; \mathbb F_2)$ as in the previous section, we will make use of two versions of equivariant homology. The defining chain complexes have underlying $\mathbb F_2[x]$-modules given by \[\widehat{C} = \widetilde C[x], \quad \overline{C} = \widetilde C\llbracket x^{-1}, x],\] with differential in both cases defined by the same formula \[d^\bullet(cx^n) = (\widetilde d c)x^n + (\chi_2 c) x^{n+1}.\] In particular, $\widehat{C}$ is a submodule of $\overline C$, and this construction is functorial with respect to $\Lambda_{\mathbb F}(\chi_2)$-module homomorphisms.

The corresponding equivariant homology groups are indicated $\widehat I_{\chi_2}(Y,w,\mathfrak a)$ and $\overline I_{\chi_2}(Y,w,\mathfrak a)$. The map induced by inclusion is an $\mathbb F[x]$-module homomorphism called the connecting homomorphism $j_{Y,w, \mathfrak a}: \widehat I \to \overline I$. These modules are functorial with respect to nice cobordisms, and the induced maps commute with $j$. Furthermore, there is an exact triangle 
\begin{equation}\label{eqn:hat-tilde}
\cdots \to \widehat I_{\chi_2}(Y,w,\mathfrak a) \xrightarrow{\cdot x} \widehat I_{\chi_2}(Y,w,\mathfrak a) \to \widetilde I_{\chi_2}(Y,w,\mathfrak a) \to \cdots 
\end{equation}

Here $\widetilde I(Y,w,\mathfrak a)$ is the homology of the complex $\widetilde C(Y,w,\mathfrak a;\mathbb F)$. 

Suppose $(Y,w)$ is equipped with abelian data $\mathfrak a = (\sigma, s)$. If $\mathfrak a' = (\sigma',s')$ has $\sigma' > \sigma$, then the cobordism $(I \times Y, I \times w): (Y,w,\mathfrak a) \to (Y, w, \mathfrak a')$ is a nice cobordism. It is proved in \cite{DMES}*{Section~8.4} that this cobordism induces an isomorphism on the equivariant homology groups, and these continuation maps compose functorially. For every two choices of abelian data on $(Y,w)$, there is another choice which is larger than both in the above sense. It follows that all three of $\widehat I_{\chi_2}(Y,w)$, $\overline I_{\chi_2}(Y,w)$, and $\widetilde I_{\chi_2}(Y,w)$ are independent of the choice of abelian data, as are the connecting homomorphism $j_{Y,w}$ and the maps in the exact triangle \eqref{eqn:hat-tilde}.

Suppose $(W,c): (Y,w) \to (Y', w')$ is nice for some choice of abelian data. This is true, for instance, if $b^+(W) = 0$ or if $[c]_2 \ne 0 \in H^2(W;\mathbb F)$. 
 
It follows that $(W,c)$ induces well-defined maps on $\widehat I_{\chi_2}$ and $\overline I_{\chi_2}$ for which the following diagram commutes: \[\begin{tikzcd}
	{\widehat I_{\chi_2}(Y, w, \mathfrak a)} & {\overline I_{\chi_2}(Y, w, \mathfrak a)} \\
	{\widehat I_{\chi_2}(Y', w', \mathfrak a')} & {\overline I_{\chi_2}(Y', w', \mathfrak a')}
	\arrow["{j_{Y,w,\mathfrak a}}", from=1-1, to=1-2]
	\arrow["{\widehat I(W,c)}"', from=1-1, to=2-1]
	\arrow["{\overline I(W,c)}", from=1-2, to=2-2]
	\arrow["{j_{Y,w,\mathfrak a}}",from=2-1, to=2-2]
\end{tikzcd}\]

\noindent Similarly, the induced maps commute with the maps in the exact triangle \eqref{eqn:hat-tilde}.

Explicit computations analogous to the work of \cite[Section 4.3]{DS1} imply first that $\widehat I_{\chi_2}(Y,w)$ is a finitely-generated $\mathbb F_2[x]$-module with free part of rank equal to $\dim Z(Y,w)$, while there is a localization isomorphism \[\Phi_{Y,w, \mathfrak a}: \overline I_{\chi_2}(Y,w,\mathfrak a) \to Z(Y,w)\llbracket x^{-1}, x].\] 

If $Z(Y,w)$ is nontrivial then $[w]_2 = 0 \in H_1(Y;\mathbb F)$, so we suppose $w = \varnothing$ for convenience and suppress it from notation. If $(W,c): (Y,\mathfrak a) \to (Y',\mathfrak a')$ is a nice cobordism, we define \[\Phi_{W,c} = \Phi_{Y',\mathfrak a'} \overline I(W,c) \Phi^{-1}_{Y, \mathfrak a}: Z(Y)\llbracket x^{-1}, x] \to Z(Y')\llbracket x^{-1},x].\] The argument of \cite[Corollary 4.12]{DS1} computes the leading term of this map.

\begin{lemma}\label{lemma:Phi-calculation}
The map $\Phi_{W,c}$ is given by multiplication by a power series \[p(x) = \epsilon_0(W,c) + \sum_{i=1}^\infty a_i x^{-i},\] where each $a_i$ is a linear map $Z(Y) \to Z(Y')$.  
\end{lemma}

In that reference, the result is stated in the case that $Z(Y), Z(Y')$ are $1$-dimensional and $\epsilon_0 = 1$, but the argument applies without these assumptions.

For the cylinder $W = I \times Y$, the map $\epsilon_0$ is the identity for any choice of abelian data. Because the $a_i$ might be nonzero, we cannot conclude that the localization map is independent of the choice of abelian data. We do find that the localization map $\Phi_{Y,w,\mathfrak a}: \overline I_{\chi_2}(Y) \to Z(Y)\llbracket x^{-1},x]$ depends on $\mathfrak a$ only up to multiplication by power series of the form $1 + \sum_{i =1}^\infty b_i x^{-i}$, where again $b_i$ denotes a linear map $Z(Y) \to Z(Y)$. In particular, there is a filtration of $\overline I_{\chi_2}$ by \textbf{degree} which is independent of any choice of abelian data, where $c \in F_r \overline I_{\chi_2}(Y)$ if $\Phi_{Y,\mathfrak a}(c)$ can be written as $\sum_{i=0}^\infty c_i x^{r-i}$ for some $c_i \in Z(Y)$. Given $c \in \overline I_{\chi_2}(Y)$, we write $\deg(c) = \text{max} \{d \mid c \in F_d \overline I_{\chi_2}(Y).\}$ 

Now consider the connecting homomorphism $j_{Y}: \widehat I_{\chi_2}(Y) \to \overline I_{\chi_2}(Y)$. Daemi and Scaduto compute that $j_Y$ is an isomorphism after localization, so in particular the induced map on $\widehat I_{\chi_2}(Y)/\text{Tors}$ is injective. 

This fact will be used later more generally, but is most important in the case that $Y$ is an $\mathbb F$-homology sphere, in which case $Z(Y) = \mathbb F$. In this case, there exists an integer $q_3 = q_3(Y) \in \mathbb Z$ so that \[\Phi_{Y,\mathfrak a}j_Y: \mathbb F_2[x] \cong \widehat I_{\chi_2}(Y)/\text{Tors} \to \overline I_{\chi_2}(Y) \cong \mathbb F_2\llbracket x^{-1}, x]\] sends $1$ to $x^{-q_3} + O(x^{-1}).$ To rephrase, we may define $q_3(Y)$ as follows.

\begin{defn}
If $Y$ is an $\mathbb F$-homology sphere, define $q_3(Y) \in \mathbb Z$ by \[q_3(Y) = \max \{-\deg(x) \mid x \in \textup{im}(j_Y)\}.\]
\end{defn}

\begin{prop}
The quantity $q_3(Y)$ is an $\mathbb F$-homology cobordism invariant which satisfies $q_3(-Y) = -q_3(Y)$.
\end{prop}
\begin{proof}
Suppose $W: Y \to Y'$ is a negative-definite cobordism between $\mathbb F$-homology spheres with $H_1(W;\mathbb F) = 0$. In this case, $\epsilon_0(W) = 1$. It follows from the above discussion that $q_3(Y) \le q_3(Y')$. Applying this to an $\mathbb F$-homology cobordism and its reverse, we deduce that $q_3(Y)$ is an $\mathbb F$-homology cobordism invariant.

The duality statement follows because $\widetilde C(Y;\mathbb F)^\vee \simeq \widetilde C(-Y;\mathbb F)$ by \cite{DMES}*{Theorem~7.3.3}. Because the $h$-invariant of a $\mathcal S$-complex negates under duality \cite{DS1}*{Proposition~4.21}, we obtain the stated formula for $q_3(-Y)$.
\end{proof}

\subsection{Framed instanton homology}\label{subsec:DMES3}
In \cite{KMframed}, Kronheimer and Mrowka introduced the homology group $I^\#(Y, w; R)$ for any pair of an oriented $3$-manifold $Y$, an oriented $1$-manifold $w \subset Y$, and commutative ring $R$. This $R$-module comes equipped with a relative $\mathbb Z/4$-grading and an absolute $\mathbb Z/2$-grading; in this article, we only discuss the absolute grading. 

To use results from the equivariant theory, we need to compare these functors. The result is simplest for rings in which $4 = 0$; we restrict attention further to $\mathbb F$. 

\begin{theorem}[{\cite[Theorem 1.3.1]{DMES}}]\label{thm:Isharp-is-Itilde}
There exists a natural isomorphism $\varphi_{Y,w}: \widetilde I(Y, w; \mathbb F) \to I^\#(Y, w; \mathbb F)$, natural with respect to the cobordism maps induced by nice cobordisms.
\end{theorem}

More generally, for an arbitrary coefficient ring, set $V = R^2$ and $\tau: V \to V$ given by $\tau(x,y) = (y,x)$. Then $I^\#(Y;R)$ is isomorphic to the fixed points of $\tau$ on \[H(\widetilde C(Y,w;\mathbb Z) \otimes V, \widetilde d + 4\chi_3 \tau).\]

This explicit description is used to provide the following computations: 

\begin{theorem}\label{thm:Brieskorn-calc}
We have $\dim I^\#(\Sigma(2,3,6k-1);\mathbb F) = 8k \pm 1$. If $(Y,w)$ is an admissible pair, then $\dim I^\#(Y,w;\mathbb F)$ is divisible by $4$. If $Y$ is a rational homology sphere with $|H_1(Y;\mathbb Z)| = 2 \mod 4$, then $\dim I^\#(Y;\mathbb F) = 2 \mod 4$.
\end{theorem}

Computations are also provided for all quotients $\SU(2)/\Gamma$. These include the surgeries $S^3_n(K)$ on the left-handed trefoil for $-5 \le n \le -1$, and the behavior in these computations partly motivated this work. 

It seems plausible that the dimension of $I^\#(Y,w;\mathbb F)$ should be divisible by $8$, and this is known to be the case when there is a genus $1$ surface $\Sigma \subset Y$ for which $w \cap \Sigma$ is odd. This would have implications for the invariants $r_2(K), M(K)$ to be defined later in this article.

\subsection{Results from \cite{DME2}}\label{subsec:DME2}
We will make use of the following statements, whose proofs are out of the scope of this paper. The following is the main result of \cite{DME2}.

\begin{theorem}
If $W: Y \to Y'$ is a cobordism between $\mathbb F$-homology spheres with $H_1(W;\mathbb F) = 0$, then 
\begin{equation}\label{eqn:q3-main-ineq}
-b^+(W) \le q_3(Y') - q_3(Y) \le b^-(W)
\end{equation}
\end{theorem}

This is proved by introducing cobordism maps associated to cobordisms which are not negative-definite, yet have $[c] = 0$. These cobordisms are not `nice' in the sense discussed in Section \ref{subsec:DMES2}. While the discussion of \cite{DME2} gives a stronger result on the chain level, we will only need the following statements in equivariant homology.

\begin{theorem}\label{thm:suspension-maps}
Suppose $(W,c): (Y, w) \to (Y', w')$ is a cobordism between rational homology spheres with $0 < b^+(W) \le 2$. Then there exist induced maps \[(W,c)_*: \widehat I_{\chi_2}(Y,w) \to \widehat I_{\chi_2}(Y',w'),\] and similarly for $\overline I_{\chi_2}$ and $\widetilde I$ with the following significance: 
\begin{enumerate}[label=(\alph*)]
    \item These induced maps are compatible with the connecting homomorphisms $j$ and the maps in \eqref{eqn:hat-tilde}.
    \item With respect to the isomorphism of Theorem \ref{thm:Isharp-is-Itilde}, the induced map $\widetilde I(W,c;\mathbb F)$ is taken to the induced map $I^\#(Y,w;\mathbb F)$.
    \item If $w, w'$ are empty, The induced map $\Phi_{W,c}: Z(Y)\llbracket x^{-1}, x] \to Z(Y')\llbracket x^{-1}, x]$ is \[\Phi_{W,c} = \epsilon(W,c) x^{b^+(W)} + O(x^{b^+(W)-1}).\]
\end{enumerate}
\end{theorem}

We do not claim, prove, or use that these maps compose functorially. 

\subsection{Surgery exact triangles}\label{subsec:SET}

We will apply two exact triangles in instanton homology. The first is due to Floer, and the version stated here appears as \cite{Scaduto}*{Theorem~2.1}.

If $(Y,K)$ is a closed oriented $3$-manifold equipped with a framed knot $K$, the surgery $Y_{p/q}(K)$ is defined by removing the neighborhood of $K$ parameterized as $D^2 \times S^1$ and regluing along the boundary by a diffeomorphism which sends the meridian $\mu$ to $p\mu + q\lambda$, where $\lambda$ is the chosen longitude. When $Y = S^3$ we equip $K$ with its Seifert framing. Write $\tilde K \subset Y_{p/q}(K)$ for the dual knot, the image of $\{0\} \times S^1$.  

\begin{theorem}\label{thm:FloerTriangle}
For any oriented $1$-manifold $w \subset Y \setminus K$, there exists an exact triangle \[\cdots \to I^\#(Y, w) \to I^\#(Y_0(K), w \cup \tilde K) \to I^\#(Y_1(K), w) \to \cdots\] The maps in this exact triangle are induced by $2$-handle cobordisms.
\end{theorem}

More precisely, there is a $2$-handle cobordism $W_0: Y \to Y_0(K)$, and this is equipped with the bundle represented by the union of $I \times w$ and the cocore of the $2$-handle, whose boundary is $\tilde K$. The cobordism $W_1: Y_0(K) \to Y_1(K)$ is again a $2$-handle cobordism, now with bundle represented by the union of $I \times w$ and the core of the $2$-handle, whose boundary is $-\tilde K$. Finally, the cobordism $W_\infty: Y_1(K) \to Y$ is again a $2$-handle cobordism, now equipped with the bundle represented by $I \times w$. \\

We will also use the distance-$2$ surgery exact triangle, first introduced in \cite{CDX}. That reference is focused on the generality of `surgery $(N+1)$-gons in $\SU(N)$-Floer homology', and is stated more generally than we will need. The following is the specialization of \cite{CDX}*{Theorem~1.6} to our situation.

\begin{theorem}\label{thm:distance2}
For any oriented $1$-cycle $w \subset Y \setminus K$, there exists an exact triangle \[\to \cdots I^\#(Y_{-1}(K),w) \to I^\#(Y,w) \oplus I^\#(Y,w \cup K) \to I^\#(Y_1(K),w \cup \tilde K) \to \cdots \]
The maps to and from the $I^\#(Y)$ factors are induced by the $2$-handle cobordisms.
\end{theorem}

To obtain this from the statement of their theorem, take $\mathcal Y$ to be the connected sum of the complement of the chosen tubular neighborhood of $K$ with $T^3$, and use the bundle given by the disjoint union of $w$ and a fixed circle in $T^3$. Writing $\mu_0, \lambda_0$ the meridian and longitude with respect to our chosen framing, in their theorem take $\lambda = \mu_0 - \lambda_0$ and $\mu = \mu_0$. The statement about cobordism maps can be extracted from the more precise statement in \cite{CDX}*{Section~6.2}.

\section{Integer surgeries}\label{sec:integer-case}
In this section, we prove Theorem \ref{thm:surgerydim} and Theorem \ref{thm:q3-surgery} in the case that $r$ is an integer. The proof involves comparing the induced maps on $\widehat I_{\chi_2}(Y,w)$ to the induced maps on $I^\#(Y,w;\mathbb F) = \widetilde I(Y,w; \mathbb F)$, and this comparison is described in Section \ref{subsec:algebra}. In Section \ref{subsec:q3-Z}, we use this to prove the key technical lemmas, which determine the value $q_3(S^3_m(K))$ on odd integer surgeries, and relate this to the dimension of $I^\#$. In Section \ref{subsec:dim-Isharp-Z} we use this to define the integers $r_2(K), M(K)$ and prove Theorem \ref{thm:surgerydim} and Theorem \ref{thm:q3-surgery}. 

\subsection{Algebraic preliminaries}\label{subsec:algebra}
Suppose $(Y,w)$ is an admissible pair. We study the finitely generated $\mathbb F_2[x]$-module $\widehat I_{\chi_2}(Y,w)$. We define $\mathbb F_2[x]$-modules by the formulas 

\begin{align*}
    T(Y,w) &= \{a \in \widehat I_{\chi_2}(Y,w) \mid x^n a = 0 \text{ for some } n \ge 0\} \\
    T^\perp(Y,w) &= \{a \in \widehat I_{\chi_2}(Y,w) \mid p(x) a = 0 \text{ for some } p(x) \in \mathbb F_2[x] \text{ with } p(0) = 1\} \\
    F(Y,w) &= \widehat I_{\chi_2}(Y,w)/\text{Tors}.
\end{align*}

When the dependence on $(Y,w)$ is clear, it will be suppressed from notation. Notice that the torsion submodule of $\widehat I_{\chi_2}$ is the direct sum of $T$ and $T^\perp$. It follows that there is a noncanonical direct sum decomposition \[\widehat I_{\chi_2} \cong T \oplus T^\perp \oplus F;\] the $T$ and $T^\perp$ summands are naturally associated to $\widehat I_{\chi_2}$, but the free summand arises from a choice of splitting.

We will be interested in relating this to $\widetilde I(Y,w;\mathbb F) \cong I^\#(Y,w;\mathbb F)$. Recall the long exact sequence \eqref{eqn:hat-tilde}, which gives rise to a short exact sequence 
\begin{equation}\label{eqn:tilde-hat-SES}
0 \to \coker(x \mid \widehat I_{\chi_2}(Y,w)) \to \widetilde I(Y,w) \to \ker(x \mid \widehat I_{\chi_2}(Y,w)) \to 0
\end{equation}
where $\coker(x \mid M)$ means the cokernel of multiplication-by-$x$ on $M$, and $\ker(x \mid M)$ is similar. 

To understand the outer terms, we first define an additional vector space, \[V(Y,w) \cong T(Y,w)/x = \coker(x \mid T(Y,w)).\] We take a moment to investigate this group, noticing that \[\ker(x \mid \widehat I_{\chi_2}(Y,w)) = \ker(x \mid T(Y,w)).\] 

\begin{lemma}\label{lemma:gr-V}
There exists a natural filtration on $V(Y,w)$ and a natural isomorphism $\ker(x \mid \widehat I_{\chi_2}(Y,w)) \cong \gr V(Y,w)$.
\end{lemma}

Here, if $M$ is a filtered module, the associated graded module $\gr M$ is defined by \[\gr M = \bigoplus_{r \in \mathbb Z} \frac{F_r M}{F_{r-1} M}.\]

\begin{proof}
This is a purely algebraic statement. Suppose $M$ is a finitely-generated $\mathbb F_2[x]$-module, and define 
\[T(M) = \bigcup_{n \ge 0} \ker(x^n \mid M), \quad V(M) = \coker(x \mid T(M)).\]
Write $\pi: T(M) \to V(M)$ for the projection map. Define a filtration on $V(M)$ as follows: $v \in F_r V(M)$ if there is some $\hat v \in T(M)$ with $\pi(\hat v) = v$ and $x^r \hat v = 0$. It is clear from the definition that this filtration is functorial in $M$. 

Define a map $\psi_M: \gr V(M) \to \ker(x \mid M)$ as follows: for $[v] \in \gr_d V(M)$, choose a representative with $x^d v = 0$, and send $\psi_M[v] = x^{d-1} v$, which is in the kernel of $x$. It is well-defined: if $[\bar v] = [v]$, then $\hat v = v + u$ where $x^{d-1} u = 0$. In particular, $x^{d-1} \hat v = x^{d-1} v$. It is straightforward to see that $\psi_M$ defines a natural transformation.

The map $\psi_M$ is injective: if $[v] \in \gr_d V(M)$ has $\pi(\hat v) = v$ and $x^{d-1} \hat v = 0$, then $v \in F_{r-1} V(M)$ and thus $[v] = 0$. The map $\psi_M$ is also surjective: if $w \in \ker(x \mid M)$, by definition $w \in T(M)$. Because $M$ is finitely-generated, there is a largest integer $d$ for which $w = x^d u$ for some $u \in T(M)$. Then $\psi_M[u] = w$. Thus, $\psi_M$ is a natural isomorphism.
\end{proof}

Next, we observe that $Z(Y,w)$ also comes equipped with a filtration, though this filtration does not come from elementary algebra. Recalling $F(Y,w) = \widehat I_{\chi_2}(Y,w)/\text{Tors}$, we have a homomorphism \[\mathfrak j_{Y,w,\mathfrak a} = \Phi_{Y,w,\mathfrak a} j_{Y,w}: F(Y,w) \to Z(Y,w)\llbracket x^{-1}, x]\] which is injective and an isomorphism after tensoring the domain with $\mathbb F\llbracket x^{-1}, x]$. While $\mathfrak j_{Y,w,\mathfrak a}$ depends on a choice of abelian data, the discussion of Section \ref{subsec:DMES2} shows that its \textit{leading order term} is well-defined. So we define $F_r Z(Y,w)$ to be the set of elements $z$ for which there exists some $a \in \widehat I_{\chi_2}(Y,w)$ with 
\[\mathfrak j_{Y,w}(a) = z x^r + O(x^{r-1}).\] 
This defines a filtration of vector spaces. Finally, every element of $Z(Y,w)$ lies in some level of this filtration: because $\mathfrak j_{Y,w}$ is an isomorphism after tensoring with $\mathbb F\llbracket x^{-1}, x]$, for any $z$ there exists some $a \in F(Y,w)$ and power series $p(x)$ so that $z = p(x) \mathfrak j_{Y,w,\mathfrak a}(a)$. Writing $p(x) = x^{-d} + O(x^{-d-1})$, we have $p(x)^{-1} z = \mathfrak j_{Y,w,\mathfrak a}(a)$, so \[\mathfrak j_{Y,w,\mathfrak a}(a) = z x^{d} + O(x^{d-1})\] and $z \in F_{d} Z(Y,w)$ as desired. 

\begin{lemma}
With respect to the filtration above, there exists an isomorphism \[\psi: \coker(x \mid F(Y,w)) \cong \gr Z(Y,w).\] The map $\epsilon_0(W,c)$ is filtered of level $-b^+(W)$, and this isomorphism takes the map induced by a cobordism $(W,c)$ to $\gr \epsilon_0(W,c).$
\end{lemma}

To say that $\epsilon_0(W,c)$ is filtered of level $-b^+(W)$ means that $F_r Z(Y,w)$ maps into $F_{r+b^+(W)} Z(Y',w')$. This induces an associated graded map sending $\gr_r$ to $\gr_{r+b^+(W)}$.

Suppose $(W,c): (Y,w, \mathfrak a) \to (Y', w', \mathfrak a')$ is either nice or has $0 < b^+(W) \le 2$. In the first case, we obtain an induced map on $\widehat I_{\chi_2}$ by the work of Section \ref{subsec:DMES2}; in the second case we obtain an induced map from Theorem \ref{thm:suspension-maps}. In either case, we obtain an induced map on $F(Y,w) = \widehat I_{\chi_2}(Y,w)/\text{Tors}$, and therefore an induced map on $\coker(x \mid F(Y,w))$. The claim is that the given isomorphism takes this induced map to $\gr \epsilon_0(W,c)$.

\begin{proof}
The isomorphism is defined as follows. Given $a \in F(Y,w)$, if $\mathfrak j_{Y,w,\mathfrak a}(a) = z x^r + O(x^{r-1})$ for some choice of abelian data $\mathfrak a$, send $\psi(a) = [z] \in \gr_r Z(Y,w).$ Observe first that this map is well-defined because the leading order term of $\mathfrak j_{Y,w,\mathfrak a}$ is independent of $\mathfrak a$. It descends to $\coker(x \mid F)$ because \[\mathfrak j_{Y,w,\mathfrak a}(ax) = z x^{r+1} + O(x^r),\] but $z \in F_r Z(Y,w)$, so $[z] = 0 \in \gr_{r+1} Z(Y,w).$ 

The resulting map is surjective because the filtration is exhaustive: if $z \in F_r Z(Y,w)$, then there exists some $a \in F(Y,w)$ with $\mathfrak j_{Y,w,\mathfrak a}(a) = zx^r + O(x^{r-1})$, and thus $\psi[a] = [z]$. Injectivity follows because $\psi$ is a surjective map between vector spaces of the same dimension. 

The map $\epsilon_0$ is filtered, and its associated graded is equal to this induced map, as a consequence of Lemma \ref{lemma:Phi-calculation} when $b^+(W) = 0$ and of Theorem \ref{thm:suspension-maps}(c) when $0 < b^+(W) \le 2$. 
\end{proof}

The map $\gr \epsilon_0(W,c)$ is of great interest to us, for the following reason.

\begin{prop}\label{prop:q3-eps}
Suppose $(W,c): Y \to Y'$ is a cobordism between $\mathbb F$-homology spheres with $\epsilon_0(W,c) = 1$. Then we have \[\gr \epsilon_0(W,c) = \begin{cases} 1 & q_3(Y') = q_3(Y) - b^+(W) \\ 0 & q_3(Y') > q_3(Y) - b^+(W). \end{cases}\]
\end{prop}

\begin{proof}
We have $F(Y) \cong \mathbb F[x]$ and $\mathfrak j_{Y}(1) = x^{-q_3(Y)} + O(x^{-q_3-1})$, and similarly that $\mathfrak j_{Y'}(1) = x^{-q_3(Y')} + O(x^{-q_3-1})$. The induced map applied to $\mathfrak j_Y(1)$ gives \[x^{-q_3(Y) + b^+(W)} + O(x^{-q_3-b^+-1}).\] As a result, we must have $-q_3(Y') \le -q_3(Y) + b^+(W)$. 

The induced map on $F/x$ is zero unless the leading order term $x^{-q_3(Y')}$ has exponent equal to $-q_3(Y) + b^+(W)$, in which case the induced map on $F/x$ is $1$. 
\end{proof}

On the other hand, we can assemble the preceding discussion into a fairly complete understanding of the induced map on $\widetilde I(W,c)$. We have a pair of natural injections \[V(Y,w) \hookrightarrow \coker(x \mid \widehat I_{\chi_2}(Y,w)) \hookrightarrow \widetilde I(Y,w;\mathbb F).\] Filter $\widetilde I(Y,w;\mathbb F)$ as \[F_r \widetilde I(Y,w;\mathbb F) = \begin{cases} V(Y,w) & r = 0 \\ \coker(x \mid \widehat I_{\chi_2}(Y,w)) & r = 1 \\ \widetilde I(Y,w;\mathbb F) & r \ge 2 \end{cases}\] Naturality of the above injections implies this filtration is natural under cobordism maps. The discussion above gave a natural isomorphism \[\gr_r \widetilde I(Y,w;\mathbb F) = \begin{cases} V(Y,w) & r = 0 \\ \gr Z(Y,w) & r = 1 \\ \gr V(Y,w) & r = 2. \end{cases}\] There exists a (noncanonical) isomorphism \[\widetilde I(Y,w) \cong V(Y,w) \oplus \gr Z(Y,w) \oplus \gr V(Y,w)\] so that, with respect to this isomorphism, \[F_r \widetilde I(Y,w) = \begin{cases} V(Y,w) \oplus 0 \oplus 0 & r = 0 \\ V(Y,w) \oplus \gr Z(Y,w) \oplus 0 & r = 1 \end{cases}\]

Because the cobordism maps respect the filtration, with respect to this direct sum decomposition the induced map $\widetilde I(W,c;\mathbb F)$ can be written as

\begin{equation}\label{eqn:triangular-matrix-form}
    \widetilde I(W,c;\mathbb F) = \begin{pmatrix} f_{11} & f_{12} & f_{13} \\ 0 & \gr \epsilon_0 & f_{23} \\ 0 & 0 & \gr f_{11}\end{pmatrix}
\end{equation}

Here, the map $f_{11} = f_{11}(W,c): V(Y,w) \to V(Y',w')$ is independent of the choice of splitting, and the bottom-right entry is its associated graded map by Lemma \ref{lemma:gr-V}. The entries $f_{12}, f_{13}, f_{23}$ depend on the choice of splitting and are difficult to control.

That the matrix is upper-triangular has some helpful consequences. If $\widetilde I(W,c)$ is injective, then so are both of the maps $f_{11}$ and $\left(\begin{smallmatrix} f_{11} & f_{12} \\ 0 & \gr \epsilon_0 \end{smallmatrix}\right)$. Similarly, if $\widetilde I(W,c)$ is surjective, then so are the maps $\gr f_{11}$ and $\left(\begin{smallmatrix} \gr \epsilon_0 & f_{23} \\ 0 & \gr f_{11} \end{smallmatrix}\right)$. This is especially useful, because if $\gr f_{11}$ is surjective then $f_{11}$ is surjective.

\subsection{The invariant $q_3$ for integer surgeries}\label{subsec:q3-Z}

We now combine the surgery sequence with the observations of the previous section to see that the behavior of $\dim I^\#(-;\mathbb F)$ and the behavior of $q_3$ determine each other. In the statements below, $w \subset S^3_m(K)$ refers to the embedded curve given by $\tilde K$. It is typically only written in the case that $m$ is even, in which case $[w]$ generates $H_1(S^3_m(K);\mathbb F)$.

\begin{lemma}\label{lemma:formula1}
    For $n \ge 0$, we have 
    \begin{align*}\dim I^\#(S^3_{-2n-1}(K);\mathbb F) &= \begin{cases} \dim I^\#(S^3_{-2n}(K),w;\mathbb F) + 1 & q_3(S^3_{-2n-1}(K)) = 0 \\ \dim I^\#(S^3_{-2n}(K),w;\mathbb F) - 1 & q_3(S^3_{-2n-1}(K)) = 1.\end{cases} \\
    \dim I^\#(S^3_{2n+1}(K);\mathbb F) &= \begin{cases} \dim I^\#(S^3_{2n}(K),w;\mathbb F) + 1 & q_3(S^3_{2n+1}(K)) = 0 \\ \dim I^\#(S^3_{2n}(K),w;\mathbb F) - 1 & q_3(S^3_{2n+1}(K)) = -1.
    \end{cases}\end{align*}
\end{lemma}

Notice that it follows from \eqref{eqn:q3-main-ineq} that $q_3(S^3_{-2n-1}(K)) \in \{0,1\}$, so these are the only two possible cases; similarly $q_3(S^3_{2n+1}(K)) \in \{-1,0\}$. 

\begin{proof}
Because $-S^3_r(K^*) \cong S^3_{-r}(K)$, the dimension of $I^\#$ is unchanged under orientation-reversal, but $q_3(-Y) = -q_3(Y)$, it suffices to compute $I^\#(S^3_{-2n-1}(K);\mathbb F)$. Throughout this argument, we apply the natural isomorphism $\widetilde I(-;\mathbb F) \cong I^\#(-;\mathbb F)$.

Apply Theorem \ref{thm:FloerTriangle} to the pair $(S^3_{-2n-1}(K), \tilde K)$, framing $\tilde K$ with meridian $ -(2n+1)\mu + \lambda$ and longitude $(2n)\mu - \lambda$. Then Floer's exact triangle specializes to \[\cdots \to I^\#(S^3;\mathbb F) \xrightarrow{I^\#(W;\mathbb F)} I^\#(S^3_{-2n-1}(K);\mathbb F) \to I^\#(S^3_{-2n}(K),w;\mathbb F) \to \cdots\] Here the cobordism $W$ is the $2$-handle cobordism equipped with $c = \varnothing$, so in particular is negative-definite with $\epsilon_0(W) = 1$. 

We use now that $\dim I^\#(S^3;\mathbb F) = 1$. We have \[\dim I^\#(S^3_{-2n-1}(K);\mathbb F) = \dim I^\#(S^3_{-2n}(K),w;\mathbb F) \pm 1,\] with positive sign if $I^\#(W;\mathbb F)\ne 0$.

Suppose first that $I^\#(W;\mathbb F) = 0$. Then the entire matrix \eqref{eqn:triangular-matrix-form} is zero, so in particular $\gr \epsilon_0(W) = 0$ and by Proposition \ref{prop:q3-eps} we have \[q_3(S^3_{-2n-1}(K)) > q_3(S^3) = 0.\] Because this value is $0$ or $1$, we conclude it is equal to $1$.  
Suppose now that $I^\#(W;\mathbb F) \ne 0$, so that its $1$-dimensional image is the kernel of the surjective homomorphism $\varphi: I^\#(S^3_{-2n-1}(K);\mathbb F) \to I^\#(S^3_{-2n}(K),w;\mathbb F)$. Apply the vector space isomorphisms 
\begin{align*}
I^\#(S^3_{-2n-1}(K);\mathbb F) &\cong V(S^3_{-2n-1}) \oplus \mathbb F \oplus \gr V(S^3_{-2n-1}) \\
I^\#(S^3_{-2n}(K),w;\mathbb F) &\cong V(S^3_{-2n},w) \oplus \gr V(S^3_{-2n},w).
\end{align*}
Comparing dimensions, we see that $\dim V(S^3_{-2n-1}) = \dim V(S^3_{-2n},w)$.

With respect to this direct sum decomposition, we write \[\varphi = \begin{pmatrix} \varphi_{11} & \varphi_{12} & \varphi_{13} \\ 
0 & 0 & \gr \varphi_{11} \end{pmatrix}.\] Because $\varphi$ is surjective, in particular $\gr \varphi_{11}$ is surjective and therefore $\varphi_{11}$ is surjective; by comparing dimensions, these are isomorphisms. It follows that any element of the kernel of $\varphi$ must have third component zero, hence first component equal to $\varphi_{11}^{-1}(\varphi_{12}(y))$, and therefore 
\[\text{span}\begin{pmatrix} \varphi_{11}^{-1} \varphi_{12}(1) \\ 1 \\ 0 \end{pmatrix} = \ker \varphi = I^\#(W;\mathbb F)(1) = \text{span} \begin{pmatrix} f_{11}(1) \\ \gr \epsilon_0(W)\\ 0 \end{pmatrix},\] so in particular $\gr \epsilon_0(W) = 1$. It now follows from Proposition \ref{prop:q3-eps} that $q_3(S^3_{-2n-1}(K)) = q_3(S^3) = 0$.
\end{proof}

This computes half of what we need. The next argument requires the use of Theorem \ref{thm:suspension-maps}, as it will involve cobordisms with $b^+(W) = 1$.

\begin{lemma}\label{lemma:formula2}
    For $n \ge 1$, we have 
    \begin{align*}\dim I^\#(S^3_{-2n+1}(K);\mathbb F) &= \begin{cases} \dim I^\#(S^3_{-2n}(K),w;\mathbb F) - 1 & q_3(S^3_{-2n+1}(K)) = 0 \\ \dim I^\#(S^3_{-2n}(K),w;\mathbb F) + 1 & q_3(S^3_{-2n+1}(K)) = 1.\end{cases} \\
    \dim I^\#(S^3_{2n-1}(K);\mathbb F) &= \begin{cases} \dim I^\#(S^3_{2n}(K),w;\mathbb F) - 1 & q_3(S^3_{2n-1}(K)) = 0 \\ \dim I^\#(S^3_{2n}(K),w;\mathbb F) + 1 & q_3(S^3_{2n-1}(K)) = -1.
    \end{cases}\end{align*}
\end{lemma}
\begin{proof}
It suffices to handle the case of positive surgeries. Apply Theorem \ref{thm:FloerTriangle} to the pair $(S^3_{2n-1}, \tilde K)$, framing $\tilde K$ with meridian $(2n-1)\mu + \lambda$ and longitude $-2n\mu - \lambda$. These give rise to the exact triangle
\[\cdots \to I^\#(S^3;\mathbb F) \xrightarrow{I^\#(W;\mathbb F)} I^\#(S^3_{2n-1}(K);\mathbb F) \to I^\#(S^3_{2n}(K),w;\mathbb F) \to \cdots \] 
Again the cobordism $W$ is equipped with surface $c = \varnothing$ and has $\epsilon_0(W) = 1$, but we now have instead that $b^+(W) = 1$. Again, we have \[\dim I^\#(S^3_{2n-1}(K);\mathbb F) = \dim I^\#(S^3_{2n}(K),w;\mathbb F) \pm 1\] with negative sign if $I^\#(W;\mathbb F) = 0$. We apply Theorem \ref{thm:suspension-maps} to study the induced map $I^\#(W;\mathbb F)$. When $I^\#(W;\mathbb F) = 0$, comparing to \eqref{eqn:triangular-matrix-form} we see that $\gr \epsilon_0 = 0$, so by Proposition \ref{prop:q3-eps} we see that \[q_3(S^3_{2n-1}(K)) > 0 - 1.\] We therefore must have $q_3(S^3_{2n-1}(K)) = 0$.

Suppose instead $I^\#(W;\mathbb F) \ne 0$, so the map \[\psi: I^\#(S^3_{2n-1}(K);\mathbb F) \to I^\#(S^3_{2n}(K),w;\mathbb F), \quad \psi = \begin{pmatrix} \psi_{11} & \psi_{12} & \psi_{13} \\ 0 & 0 & \gr \psi_{11}\end{pmatrix}\] is surjective with one-dimensional kernel. Arguing exactly as before, its kernel is spanned by an element of the form $\begin{pmatrix} x \\ 1 \\ 0 \end{pmatrix}$, so the induced map on $I^\#(W;\mathbb F)$ must have this element in its image. Applying Theorem \ref{thm:suspension-maps} to compare this to the induced map on $\widehat I_{\chi_2}$, we see $\gr \epsilon_0(W) \ne 0$ exactly as before. Applying Proposition \ref{prop:q3-eps}, we see that \[q_3(S^3_{2n-1}(K)) = q_3(S^3) - 1 = -1.\qedhere\]
\end{proof}

When $|\dim I^\#(S^3_{2n+1};\mathbb F) - \dim I^\#(S^3_{2n-1};\mathbb F)| = 2$, the dimension of $I^\#(S^3_{2n};\mathbb F)$ is the unique number lying between these, and the same is true for $I^\#(S^3_{2n},w;\mathbb F)$. We now determine the behavior in the remaining case, where $\dim I^\#(S^3_{2n \pm 1})$ are equal. 

\begin{prop}\label{prop:even-behavior-w}
    Suppose that $\dim I^\#(S^3_{2n-1}(K);\mathbb F) = \dim I^\#(S^3_{2n+1}(K);\mathbb F)$. Then we have \[\dim I^\#(S^3_{2n}(K),w;\mathbb F) = \dim I^\#(S^3_{2n \pm 1}(K);\mathbb F) - 1.\] 
\end{prop}
\begin{proof}
Suppose not, so instead $\dim I^\#(S^3_{2n},w) = \dim I^\#(S^3_{2n \pm 1}) + 1$. We handle separately the cases $n = 0$ and $n > 0$; the case $n < 0$ follows from duality.

Suppose first that $n > 0$. By Lemma \ref{lemma:formula1} we have $q_3(S^3_{2n+1}(K)) = -1$, while by Lemma \ref{lemma:formula2} we have $q_3(S^3_{2n-1}(K)) = 0$. However, the composite of the $2$-handle cobordisms gives a simply-connected negative-definite cobordism $W: S^3_{2n-1}(K) \to S^3_{2n+1}(K)$, so $q_3(S^3_{2n-1}(K)) \le q_3(S^3_{2n+1}(K))$, giving a contradiction in this case.

If $n = 0$, the analogous argument gives $q_3(S^3_{-1}(K)) = 1$ and $q_3(S^3_1(K)) = -1$. However, the composite of $2$-handle cobordisms gives a simply-connected cobordism $W: S^3_{-1}(K) \to S^3_1(K)$ with $b^+(W) = b^-(W) = 1$, so \eqref{eqn:q3-main-ineq} gives \[-1 \le q_3(S^3_1(K)) - q_3(S^3_{-1}(K)) \le 1,\] a contradiction.
\end{proof}

The discussion is somewhat more complicated when $S^3_{2n}(K)$ is equipped with the trivial bundle. We handle here the case $n \ne 0$ and the case $n = 0$ in the next section.

\begin{prop}\label{prop:even-behavior}
    Suppose that $n \ne 0$ and $\dim I^\#(S^3_{2n-1}(K);\mathbb F) = \dim I^\#(S^3_{2n+1}(K);\mathbb F)$. Then we have \[\dim I^\#(S^3_{2n}(K);\mathbb F) = \dim I^\#(S^3_{2n \pm 1}(K);\mathbb F) + 1.\]
\end{prop}

\begin{proof}
By duality we may suppose $n > 0$. Proposition \ref{prop:even-behavior-w}, combined with Lemma \ref{lemma:formula1} and Lemma \ref{lemma:formula2}, gives that $q_3(S^3_{2n-1}(K)) = -1$ and $q_3(S^3_{2n+1}(K)) = 0$.  

Towards a contradiction, suppose $\dim I^\#(S^3_{2n}(K);\mathbb F) = \dim I^\#(S^3_{2n \pm 1}(K);\mathbb F) - 1$. Decompose the vector spaces of interest as 
\begin{align*}
I^\#(S^3_{2n-1}(K);\mathbb F) &\cong V_{-1} \oplus \mathbb F \oplus \gr V_{-1} \\ 
I^\#(S^3_{2n}(K);\mathbb F) &\cong V_{0} \oplus \mathbb F^2 \oplus \gr V_{0} \\ 
I^\#(S^3_{2n+1}(K);\mathbb F) &\cong V_1 \oplus \mathbb F \oplus \gr V_1
\end{align*}
where $\dim V_{-1} = \dim V_1 = d$ and $\dim V_0 = d-1$. 

We study the two exact triangles obtained by applying Theorem \ref{thm:FloerTriangle} to the pairs $(S^3_{2n}(K), \tilde K)$ and $(S^3_{2n+1}(K), \tilde K)$ with standard framing. We obtain exact triangles 
\begin{align*}
    \cdots \to I^\#(S^3_{2n-1}(K);\mathbb F) \xrightarrow{I^\#(W)} &I^\#(S^3_{2n}(K);\mathbb F) \to I^\#(S^3;\mathbb F) \to \cdots \\
    \cdots \to I^\#(S^3_{2n}(K);\mathbb F) \xrightarrow{I^\#(W')} &I^\#(S^3_{2n+1}(K);\mathbb F) \to I^\#(S^3;\mathbb F) \to \cdots
\end{align*}
where both $W$ and $W'$ are simply-connected and negative-definite, with \[\epsilon_0 = \epsilon_0(W) = \begin{pmatrix} 1 \\ 0 \end{pmatrix} \quad \epsilon'_0 = \epsilon_0(W') = \begin{pmatrix} 1 & 0 \end{pmatrix}\]

We will show that the composite has $\gr \epsilon_0(W \cup W') = 1$ by showing that each of $\gr \epsilon_0$ and $\gr \epsilon'_0$ are nonzero.

By assumption, $I^\#(W)$ is surjective and $I^\#(W')$ is injective. Write these in matrix form as 
\[I^\#(W) = \begin{pmatrix} f_{11} & f_{12} & f_{13} \\ 0 & \gr \epsilon_0 & f_{23} \\ 0 & 0 & \gr f_{11} \end{pmatrix}, \quad I^\#(W') = \begin{pmatrix} g_{11} & g_{12} & g_{13} \\ 0 & \gr \epsilon'_0 & g_{23} \\ 0 & 0 & \gr g_{11} \end{pmatrix}.\] Because $I^\#(W)$ is surjective and $I^\#(W')$ is injective, the two submatrices \[\begin{pmatrix} \gr \epsilon_0 & f_{23} \\ 0 & \gr f_{11} \end{pmatrix}, \quad \begin{pmatrix} g_{11} & g_{12} \\ 0 & \gr \epsilon'_0 \end{pmatrix}\] are surjective and injective, respectively. By counting dimensions, they must in fact be isomorphisms, which implies $\gr \epsilon_0$ and $\gr \epsilon'_0$ are nonzero.

We arrive at our contradiction by applying Proposition \ref{prop:q3-eps} to the composite, giving that $q_3(S^3_{2n-1}(K)) = q_3(S^3_{2n+1}(K))$.
\end{proof}

\begin{remark}
The argument involves cobordism maps to and from $S^3_{2n}(K)$, which have only been established in the equivariant theory when $n \ne 0$. It seems likely that this argument would apply with minimal change if the equivariant theory had an extension to the case where $Y$ has $b_1(Y) > 0$ and is equipped with a trivial bundle.
\end{remark} 

We will handle the case $n = 0$ in the next section. 

\subsection{The dimension of $I^\#$ for integer surgeries}\label{subsec:dim-Isharp-Z}

We first nearly prove Theorem \ref{thm:surgerydim} and Theorem \ref{thm:q3-surgery} for integer surgeries, leaving some ambiguity in the case of $0$-surgery.

\begin{prop}\label{prop:mainthm-Z}
For every knot $K$ there are even integers $r_2(K), M(K)$ and $\zeta(K) \in \{0, 1\}$ so that for all integers $n$, we have \[\dim I^\#(S^3_n(K), w; \mathbb Z) = \begin{cases} r_2(K) + 2 -2\zeta(K) & n = M(K) \text{ and } [w] = 0 \\ r_2(K) + |n-M(K)| & \text{otherwise} \end{cases}\] If $M(K) \ne 0$ then $\zeta(K) = 0$, and in any case $M(K)$ is characterized by the formula, for odd integers $m$, \[q_3(S^3_m(K)) = \begin{cases} 1 & M(K) < m < 0 \\ -1 & 0 < m < M(K) \\ 0 & \text{otherwise.} \end{cases}\] 
\end{prop}

We will soon show that $\zeta(K) = 0$ for all knots $K$. 

\begin{proof}
The strategy of proof will be as follows. We will show that there is a unique integer $n$ for which $\dim I^\#(S^3_{2n+1}(K);\mathbb F) = \dim I^\#(S^3_{2n-1}(K);\mathbb F)$, and for other $n'$ we have \[\dim I^\#(S^3_{2n'+1}(K);\mathbb F) - \dim I^\#(S^3_{2n'-1}(K);\mathbb F) = \begin{cases} 2 & n' > n \\ 0 & n' = n \\ -2 & n' < n.\end{cases}\] Write $M(K) = 2n$. Because dimensions of $I^\#$ on adjacent integer slopes differ by $\pm 1$, the above implies \[\dim I^\#(S^3_m(K),w;\mathbb F) - \dim I^\#(S^3_{m-1}(K),w;\mathbb F) = \begin{cases} 1 & m > M(K) + 1 \\ -1 & m < M(K)\end{cases}\] for either choice of bundle.

The behavior of this difference for $m = M(K)$ and $m = M(K)-1$ is determined by Proposition \ref{prop:even-behavior-w} and Proposition \ref{prop:even-behavior}. In the case of possibly non-trivial bundles, we find \[\dim I^\#(S^3_m(K),w;\mathbb F) - \dim I^\#(S^3_{m-1}(K),w;\mathbb F) = \begin{cases} 1 & m > M(K) \\ -1 & m \le M(K)\end{cases}.\] 

If we set $r_2(K) = \dim I^\#(S^3_{M(K)},w;\mathbb F)$, this gives the stated dimension formula in the case of non-trivial bundle. Except when $M(K) = 0$, we showed in Proposition \ref{prop:even-behavior} that $\dim I^\#(S^3_{M(K)};\mathbb F) = r_2(K) + 2$ and thus $\zeta(K) = 0$. In the case $M(K) = 0$, the dimension is at this point undetermined between $r_2(K)$ and $r_2(K) + 2$.

It remains to show the claim before regarding the behavior on odd surgeries. Let $n$ be an integer for which $\dim I^\#(S^3_{2n\pm 1}(K);\mathbb F)$ are equal; the computations above imply $\dim I^\#(S^3_{2n}(K),w;\mathbb F)$ is one less.

If $n = 0$, then by Lemma \ref{lemma:formula1} and Lemma \ref{lemma:formula2} we have $q_3(S^3_{\pm 1}(K)) = 0$. Because there are simply connected negative-definite cobordisms \[S^3_1(K) \to S^3_m(K) \to S^3 \to S^3_{-m}(K) \to S^3_{-1}(K),\] for each $m > 0$, it follows from \eqref{eqn:q3-main-ineq} that $q_3(S^3_m(K)) = 0$ for all odd integers $m$. Conversely, if this is the case, then $q_3(S^3_{\pm 1}(K)) = 0$ and the dimensions are equal when $n = 0$. 

If $n > 0$, Lemma \ref{lemma:formula1} and Lemma \ref{lemma:formula2} give instead $q_3(S^3_{2n-1}(K)) = -1$ and $q_3(S^3_{2n+1}(K)) = 0$. Similar to the discussion above, this is equivalent to $q_3(S^3_m(K)) = -1$ for $0 < m < 2n$ odd and $q_3(S^3_m(K)) = 0$ for $2n < m$. 

If $n < 0$, Lemma \ref{lemma:formula1} and Lemma \ref{lemma:formula2} give $q_3(S^3_{2n-1}(K)) = 0$ and $q_3(S^3_{2n+1}(K)) = 1$, which is then equivalent to $q_3(S^3_m(K)) = 0$ for $m < 2n$ and is equal to $1$ for $2n < m < 0$. 

It is clear that the case $n = 0$ is mutually exclusive from the other two. The cases $n > 0$ and $n < 0$ are also mutually exclusive, because $q_3(S^3_{\pm 1}(K)) = \mp 1$ is impossible: there is a simply-connected cobordism $W: S^3_{-1}(K) \to S^3_1(K)$ with $b^+(W) = b^-(W) = 1$. 

So there is at most one such integer $n$. We must argue that one exists. If $q_3(S^3_1(K)) = -1$, we must have $q_3(S^3_m(K)) = 0$ for some $m > 0$. Otherwise, $\dim I^\#(S^3_m(K);\mathbb F)$ would be a decreasing sequence in odd integers $m$, but this quantity is no smaller than $m$. A similar argument applies in the case $q_3(S^3_{-1}(K)) = 1$. If neither of these are true, then $M(K) = 0$. 
\end{proof}

\begin{example}\label{ex:trefoil}
    For the left-handed trefoil knot $3_1$, the manifold $S^3_{-5}(K)$ is a lens space and $S^3_{-1}(K)$ is the Poincar\'e homology sphere. We have $\dim I^\#(S^3_{-5}(K);\mathbb F) = 5$ and $\dim I^\#(S^3_{-1}(K);\mathbb F) = 7$ by Theorem \ref{thm:Brieskorn-calc}. Given these computations, the formula from Proposition \ref{prop:mainthm-Z} implies that $M(3_1) = -4$, while $r_2(3_1) = 4$ and $\zeta(K) = 0$. 
\end{example}

The following is a consequence of the characterization in Proposition \ref{prop:mainthm-Z}.

\begin{prop}\label{prop:or-rev}
The quantity $M(K)$ is an invariant of the concordance class of $K$, and the mirror $K^*$ satisfies $M(K^*) = -M(K)$. 
\end{prop}
\begin{proof}
If $K$ is concordance to $K'$, then for any odd integer $m$ the manifold $S^3_m(K)$ is $\mathbb F$-homology cobordant to $S^3_m(K')$, and therefore they have the same $q_3$. Thus $M(K) = M(K')$. The mirror has $S^3_{-m}(K^*) = -S^3_m(K)$, so applying $q_3(-Y) = -q_3(Y)$ we obtain $M(K^*) = - M(K)$.
\end{proof}

Write $F_{K,n}: I^\#(S^3;\mathbb F) \to I^\#(S^3_n(K), w; \mathbb F)$ for the map induced by the $2$-handle cobordism equipped with the cocore of the handle attachment. As a consequence of the surgery exact sequence and $\dim I^\#(S^3;\mathbb F) = 1$, we see that 
\begin{align*}F_{K,n} \text{ nonzero } &\iff I^\#(S^3_{n+1}(K);\mathbb F) = \dim I^\#(S^3_n(K);\mathbb F) - 1\\
&\iff n \le M(K) - 2 \text{ or } n = M(K) - \zeta(K).\end{align*}

\begin{lemma}\label{lemma:M-add}
We have $M(K \# L) = M(K) + M(L)$, and if we have $\zeta(K) = 0$ while $\zeta(L) = 1$ then $\zeta(K \# L) = 1$. 
\end{lemma}
\begin{proof}
It is established in \cite[Lemma 5.1]{bs-concordance} that if $F_{K, a}$ and $F_{L,b}$ are nonzero, the same is true for $F_{K \# L, a + b}$. Write $S = \{n \in \mathbb Z \mid F_{K \# L, n} \ne 0\}$. This set takes the form $(-\infty, M(K \# L) - 2] \cup \{M(K \# L) - \zeta(K \# L)\}$. We study this set in cases. 
\begin{itemize}
    \item If $\zeta(K) = \zeta(L) = 0$, then $S$ contains $(-\infty, M(K) + M(L) - 2]$, so $M(K \# L) \ge M(K) + M(L)$. 
    \item If $\zeta(K) = 0$ and $\zeta(L) = 1$ or vice versa, then $S$ contains $(-\infty, M(K) + M(L) - 1]$, so that $M(K \# L) \ge M(K) + M(L)$. 
    \item If $\zeta(K) = \zeta(L) = 1$, then $S$ contains $(-\infty, M(K) + M(L) -2] \cup \{M(K) + M(L)\}$, so that $M(K \# L) \ge M(K) + M(L)$. 
\end{itemize}
In all cases, we see that $M(K \# L) \ge M(K) + M(L)$. Proposition \ref{prop:or-rev} gives \[M(K \# L) = -M(K^* \# L^*) \le -M(K^*) - M(L^*) = M(K) + M(L),\] so that $M$ is additive under connected sum. Returning to the list above, we see that when $\zeta(K) = 0$ and $\zeta(L) = 1$, the set $S$ contains $(-\infty, M(K) + M(L) - 1]$. Because $M(K \# L) = M(K) + M(L)$, this implies $\zeta(K \# L) = 1$. 
\end{proof}

We now show that $\zeta(K) = 0$ holds even in the case $M(K) = 0$.

\begin{lemma}
    We have $\zeta(K) = 0$ for all $K$. 
\end{lemma}
\begin{proof}
Recall from Example \ref{ex:trefoil} that $M(3_1) = -4$ and $\zeta(3_1) = 0$. Suppose $M(K) = 0$. Since $M(3_1) = M(K \# 3_1) = -4$, we have $\zeta(3_1) = \zeta(K \# 3_1) = 0$. It follows from the preceding lemma that $\zeta(K) = 0$.
\end{proof}

We have therefore computed $\dim I^\#(S^3_n(K),w;\mathbb F)$ for all integers and all choices of bundle. We conclude with some comments on parity.

\begin{lemma}
    We have $r_2(K) \ge |M(K)|$, while $r_2(K)$ and $M(K)$ are both divisible by $4$.
\end{lemma}
\begin{proof}
    Because \[r_2(K) + |n - M(K)| = \dim I^\#(S^3_n(K),w;\mathbb F) \ge |n|,\] taking $n$ large implies $r_2(K) \ge M(K)$ while taking $n$ small implies $r_2(K) \ge -M(K)$.

    Combining Theorem \ref{thm:Brieskorn-calc} and Theorem \ref{thm:surgerydim} we have \[4 \mid \dim I^\#(S^3_0(K),w;\mathbb F) = r_2(K) + |M(K)|.\] Thus, $r_2 \equiv -|M| \mod 4$. Because $M(K)$ is even, we find $r_2(K) \equiv M(K) \mod 4$. 

    Finally, suppose towards a contradiction that $M(K) = 4m+2$ for some integer $m$. Then $\dim I^\#(S^3_{M(K)};\mathbb F) = r_2(K) + 2 \equiv 0 \mod 4$. However, because $|H_1(S^3_{M(K)};\mathbb Z)| = 2 \mod 4$, this contradicts Theorem \ref{thm:Brieskorn-calc}.
\end{proof}

This proves Theorem \ref{thm:surgerydim} and Theorem \ref{thm:q3-surgery} for integer surgeries. Theorem \ref{thm:q3-surgery} implies that $M(K)$ is a concordance invariant. Combining the preceding lemma with Lemma \ref{lemma:M-add}, we see that $M$ defines a homomorphism $\mathcal C \to 4\mathbb Z$. It is surjective by Example \ref{ex:trefoil}, so this proves Theorem \ref{thm:M-additive}.

\section{Rational surgeries}\label{sec:rat}
In this section, we extend the previous computations to all rational surgeries. 

\subsection{Preliminaries}
We will prove Theorem \ref{thm:surgerydim} inductively by giving upper and lower bounds for $\dim I^\#(S^3_r(K);\mathbb F)$ and checking these bounds coincide. The upper bound comes from Theorem \ref{thm:FloerTriangle}, and the lower bound from Theorem \ref{thm:distance2}, and the following principle: given an exact triangle \[\cdots \to V_1 \to V_2 \to V_3 \xrightarrow{f} V_1 \to \cdots,\] we have \begin{align*}\dim V_2 &= \dim V_1 - \dim V_3 + 2\nll(f) \\ &= \dim V_3 - \dim V_1 + 2 \rank(f) \\ &= \dim V_1 + \dim V_3 - 2 \rank(f),\end{align*} so in particular \[|\dim V_1 - \dim V_3| \le \dim V_2 \le \dim V_1 + \dim V_3.\]

The next statement is the key to the inductive step in our argument. Most of this is described in the statement of \cite{bs-concordance}*{Proposition~4.23} and is derived there from the arithmetic of continued fractions; the remaining statements below are derived in the course of the proof given in that reference.

\begin{lemma}
Suppose $(p,q)$ is a coprime pair of integers with $q \ge 2$. Then there exists a triple $(a,b), (c, d), (e,f)$ of coprime pairs of integers with the following properties: 
\begin{itemize}
    \item $qc -pd = bc - ad = pb - qa = 1$. 
    \item $b, d > 0$ and $f \ge 0$ with equality if only if $e = 1$. 
    \item If $n$ is an integer with $n < p/q < n+1$, then $a/b, c/d \in [n, n+1]$ and $e/f$ lies in this interval unless $(e,f) = (1,0)$.
    \item $(p,q) = (a,b) + (c,d)$. If $b = d$, then both are $1$ and $(e,f) = (1,0)$. If $b < d$ then $(e,f) = (c-a, d-b)$, while if $b > d$ then $(e,f) = (a-c, b-d)$.
\end{itemize}
\end{lemma}

Notice that if $(p,q)$ has triple $(a,b), (c, d), (e,f)$, then a triple for $(-p,q)$ is given by $(-p, q), (-c,d), (-a,b), (-e,f)$. We now describe the two relevant exact triangles. 

Apply Theorem \ref{thm:FloerTriangle} to the triple $(S^3_{c/d}(K), \tilde K, w_0)$ where $\tilde K$ is framed with meridian $c\mu + d\lambda$ and longitude $a \mu + b \lambda$, and $w_0 \subset S^3 \setminus K$ is a $1$-manifold to be specified later. The resulting exact triangle takes the form 
\begin{equation}\label{eqn:Q-Floer-triangle}
\cdots \to I^\#(S^3_{c/d}(K), w_0) \to I^\#(S^3_{a/b}(K), w_0 + \tilde K) \to I^\#(S^3_{p/q}(K), w_0) \to \cdots 
\end{equation} 
Also apply Theorem \ref{thm:distance2} to a triple $(S^3_{a/b}(K), \tilde K, w_0)$ with $\tilde K$ framed with meridian $a \lambda + b \mu$ and longitude $-(c\mu + d\lambda)$. The resulting exact triangle takes the form

\begin{equation}\label{eqn:Q-CDX-triangle}
\begin{tikzcd}
	I^\#(S^3_{e/f}(K), w_0) \arrow[r] & I^\#(S^3_{a/b}(K), w_0) \oplus I^\#(S^3_{a/b}(K), w_0 + \tilde K) \arrow[d] \\
	& I^\#(S^3_{p/q}(K), w_0 + \tilde K) \arrow[ul]
\end{tikzcd}
\end{equation}

\subsection{Proof of Theorem \ref{thm:surgerydim}}

To minimize case analysis in the proof, we first handle surgeries involving the slope $M(K)$ and half-integer surgeries.  

\begin{lemma}\label{lemma:M-plus-1n}
Theorem \ref{thm:surgerydim} holds for all knots $K$ and all slopes $M(K) \pm 1/n$ for integers $n \ge 1$.
\end{lemma}

\begin{proof}
By a duality argument it suffices to handle the case $M(K) - 1/n$. We aim to show that the dimension of $I^\#(S^3_{M-1/n};\mathbb F)$ is $nr_2 + 1$. We prove the claim by induction on $n$, where the case $n = 1$ is known. Take $(p,q) = ((n+1)M-1,n+1)$ with triple $(a,b) = (nM-1,n), (c,d) = (M, 1)$, and $(e,f) = ((n-1)M-1,n-1)$. 

Applying \eqref{eqn:Q-Floer-triangle} with $w_0 = \mu$ gives the inequality 
\begin{align*}\dim I^\#(S^3_{M-1/(n+1)}(K);\mathbb F) &\le (nr_2 + 1) + r_2 = (n+1)r_2 + 1.
\end{align*}  Applying \eqref{eqn:Q-CDX-triangle} with $w_0 = \varnothing$ gives 
\begin{align*}
    \dim I^\#(S^3_{M-1/(n+1)}(K); \mathbb F) &\ge 2(nr_2 + 1) - ((n-1)r_2 + 1) = (n+1)r_2 + 1.
\end{align*} 
Thus the lower and upper bounds agree, and give the desired result.
\end{proof}

In handling the general case, we will want some more control over certain cobordism maps involving $S^3_M(K)$. Write $W_n: I^\#(S^3_M(K);\mathbb F) \to I^\#(S^3_{M+1/n}(K);\mathbb F)$ for the cobordism map induced by the $2$-handle cobordism equipped with trivial bundle, and $W'_n$ for the cobordism map equipped with nontrivial bundle dual to the cocore.

Similarly, write $X: I^\#(S^3;\mathbb F) \to I^\#(S^3_M(K);\mathbb F)$ for the cobordism map induced by the $2$-handle cobordism with trivial bundle and $X'$ for the cobordism map equipped with bundle dual to the core. 

\begin{lemma}\label{lemma:nplus1-cob}
We have $\ker(W_n) = \im(X')$ and $\ker(W'_n) = \im(X)$ for all $n \ge 1$. 
\end{lemma}
\begin{proof}
The argument for $W'_n$ is similar, so we focus on the case of $W_n$. 

For $n=1$, this follows from Floer's exact triangle, which gives exactness of \[\cdots \to I^\#(S^3;\mathbb F) \xrightarrow{X'} I^\#(S^3_M(K);\mathbb F) \xrightarrow{W_1} I^\#(S^3_{M+1}(K);\mathbb F) \to \cdots \]

We now prove the claim in general by induction. Suppose $n \ge 1$. Consider the (non-commuting) diagram 

\[\begin{tikzcd}
	& {I^\#(S^3_{M+1/n}(K);\mathbb F)} & {I^\#(S^3_{M+1/(n+1)}(K);\mathbb F)} \\
	{I^\#(S^3_{M+1/(n-1)}(K);\mathbb F)} & {I^\#(S^3_{M}(K);\mathbb F)}
	\arrow[from=1-2, to=2-1]
	\arrow["{\bar W'_n}", shift left=3, from=1-2, to=2-2]
	\arrow[from=1-3, to=1-2]
	\arrow["{\bar W'_{n-1}}"', from=2-1, to=2-2]
	\arrow["{W_n}", from=2-2, to=1-2]
	\arrow["{W_{n+1}}"', from=2-2, to=1-3]
\end{tikzcd}\]

We have $\ker(W_{n+1}) = \im(\bar W_n')$. As discussed in \cite[Page 23]{bs-concordance}, the composite $W_n \circ \bar W'_n$ contains a sphere of self-intersection number zero which intersects whose intersection with the surface defining the bundle data is $a = nM+1 \equiv 1 \mod 2$, and as discussed in \cite[Lemma 3.10]{LiYe1}, this implies the composite map is zero.

Therefore, \[\ker(W_{n+1}) = \im(\bar W_n') \subset \ker(W_n)\] for all $n \ge 1$. Finally, by comparing dimensions in the exact triangle using Lemma \ref{lemma:M-plus-1n}, we find \[(n-1)r_2 + 1 = nr_2 + 1 - (r_2 + 2) + 2 \nll(W_n)\] so that $\nll(W_n) = 1$ for all $n$, so all containments are equalities.
\end{proof}

\begin{lemma}\label{lemma:half-surgery}
Theorem \ref{thm:surgerydim} holds for all slopes $n + 1/2$ with $n \in \mathbb Z$.
\end{lemma}
\begin{proof}
We handled the case of $M \pm 1/2$ in the previous statement. It suffices to assume $n > M(K)$, as the other case is handled by duality: if $n < M(K)-1$ we have $n + 1/2 = -[-(1+n) + 1/2]$ where $-(1+n) > M(K^*)$.

Now take $(p,q) = (2n+1, 2)$ with $(a,b) = (n,1)$ and $(c,d) = (n+1,1)$ and $(e,f) = (1,0)$. Applying \eqref{eqn:Q-Floer-triangle} with $w_0 = \varnothing$ gives the inequality 
\[\dim I^\#(S^3_{(2n+1)/2}(K);\mathbb F) \le (r_2 + |n+1-M|) + (r_2+|n-M|) = 2r_2 + |2n+1-2M|,\] here using that $n > M$ so these absolute values may be combined. Applying \eqref{eqn:Q-CDX-triangle} with $w_0 = \varnothing$ gives 
\begin{align*}
    \dim I^\#(S^3_{(2n+1)/2}(K); \mathbb F) &\ge 2(r_2+|n+1-M|) - 1 = 2r_2 + |2n+1-2M|.
\end{align*} 

Thus the lower and upper bounds agree and give the desired result. 
\end{proof}

We finally handle the remaining cases. 

\begin{proof}[Proof of Theorem \ref{thm:surgerydim}]
We have established \eqref{eqn:dim-formula} for all knots, all choices of bundle, and all slopes which either have denominator $q \le 2$ or take the form $M \pm 1/n$. 

Suppose now that \eqref{eqn:dim-formula} holds for all rational numbers whose (positive) denominator is less than $q$. We aim to prove it for slope $p/q$ and for all knots $K$. Without loss of generality, we may suppose $(p,q)$ has a triple $(a,b), (c, d), (e,f)$ so that $b > d$: the case $b = d$ was handled by Lemma \ref{lemma:half-surgery}, and if $b < d$, replace this with a triple corresponding to $(-p,q)$. The result for $p/q$ then follows by a duality argument.\\

We first prove the dimension formula in the case that $S^3_{p/q}(K)$ is equipped with the trivial bundle. In \eqref{eqn:Q-Floer-triangle}, take $w_0 = \varnothing$ if $a$ or $p$ is even, while $w_0 = \mu$ if $c$ is even. This ensures that $S^3_{p/q}$ is equipped with a trivializable bundle, while the other two are equipped with the non-trivial bundle if it is available. Because $b+d=q$ and $b,d$ are positive, we see that $b,d < q$, so we may apply the inductive hypothesis to obtain an upper bound \[\dim I^\#(S^3_{p/q}(K);\mathbb F) \le br_2 + |a-bM| + dr_2 + |c-dM|.\] Because $a/b$ and $c/d$ lie in the same integer interval, $a-bM$ and $c-dM$ have the same sign, so these expressions may be summed to give \[\dim I^\#(S^3_{p/q}(K);\mathbb F) \le (b+d)r_2 + |(a+c) - (b+d)M| = qr_2 + |p-qM|.\]
In \eqref{eqn:Q-CDX-triangle}, take $w_0 = \varnothing$ for $p$ odd and $w_0 = -\mu$ for $p$ even, so as to ensure that the $S^3_{p/q}$ term is equipped with a trivializable bundle. If $p$ is even, then $S^3_{e/f}$ is equipped with a non-trivial bundle. By hypothesis we have $b > d > 0$, so $a/b \ne M$. Our dimension formula gives the inequality \[\dim I^\#(S^3_{p/q}(K);\mathbb F) \ge 2(br_2 + |a-bM|) - (fr_2 + |e-fM|).\] Because $b > d$ we have $(e,f) \ne (1,0)$, so $a-bM$ and $e-fM$ have the same sign and these may be combined to $(2b-f) r_2 + |2a-e - (2b-f)M| = qr_2 + |p-qM|$.\\

We now handle the case where $p$ is even and $S^3_{p/q}$ is equipped with a nontrivial bundle. The same application of Floer's triangle, now with $w_0 = -\mu$, gives the desired upper bound \[\dim I^\#(S^3_{p/q}(K),w;\mathbb F) \le qr_2 + |p-qM|.\] 
Apply also \eqref{eqn:Q-CDX-triangle} with $w_0 = \varnothing$. As before, we have $a/b \ne M$. If $e/f \ne M$, then the lower bound goes through exactly as above. However, in the case $e/f = M$ additional care is required. In this case, because $(a,b) = (c,d) + (e,f)$ and $ad-bc = 1$, we find 
\[(c,d) = ((b-1)M+1,b-1), \quad (a,b) = (bM+1,b), \quad (p,q) = ((2b-1)M+2, 2b-1).\] 

The exact triangle takes the shape
\[\cdots \to I^\#(S^3_{M+2/(2b-1)}(K),w;\mathbb F) \to I^\#(S^3_M(K);\mathbb F) \xrightarrow{W_b \oplus W'_b} I^\#(S^3_{M+1/b}(K);\mathbb F)^2 \to \cdots\]
Here $W_b, W'_b$ are induced by the same $2$-handle cobordism, but where the former is equipped with trivial bundle and the latter with bundle dual to the cocore. We have at this point that \[2(br_2+1) - (r_2 + 2) + 2 \nll(W_b \oplus W'_b) = \dim I^\#(S^3_{M+2/(2b-1)}(K),w;\mathbb F) \le (2b-1) r_2 + 2\] and we aim to show $\nll(W_b \oplus W'_b) = 0$. By Lemma \ref{lemma:nplus1-cob}, we see that $\ker(W_b) = \im(X')$ and $\ker(W'_b) = \im(X)$ for $X: S^3 \to S^3_M(K)$ the $2$-handle cobordism equipped with either the trivial bundle or the bundle dual to the core. To conclude, it suffices to show $\im(X) = \im(X')$. To see this, consider the distance-$2$ exact triangle 
\[\cdots \to I^\#(S^3_{M-2}(K);\mathbb F) \to I^\#(S^3;\mathbb F)^2 \xrightarrow{X \oplus X'} I^\#(S^3_M(K);\mathbb F) \to \cdots\] We have \[2\rank(X \oplus X') = 2 + (r_2 + 2) - (r_2 + 2) = 2.\] Because both maps have image of dimension $1$, it follows that $\im(X) = \im(X')$ as desired.\qedhere
\end{proof}

\subsection{Proof of Theorem \ref{thm:q3-surgery}}

It suffices to prove this claim for positive slopes. We use two general facts: 
\begin{itemize} 
\item if $0 < r < r'$ or $r < r' < 0$ are rational, there exists a cobordism $W: S^3_r(K) \to S^3_{r'}(K)$ which is negative-definite and has $H_1(W;\mathbb F) = 0$ \cite{OwensStrle-negdef}*{Lemma 2.6}.
\item if $r, r'$ are any nonzero rationals, there exists a cobordism $W: S^3_r(K) \to S^3_{r'}(K)$ with $H_1(W;\mathbb F) = 0$ and $b^+(W) = b^-(W) = 1$.
\end{itemize}

The first bullet point implies that $q_3(S^3_r(K))$ is increasing as a function of $r$ on the negative rationals and increasing as a function of $r$ on the postive rationals. The second bullet point implies that the values on the positive rationals are in $\{-1, 0\}$ and on the negative rationals are in $\{0, 1\}$ and at most one of $\pm 1$ can be realized. In particular, if $q_3(S^3_r(K)) = 1$ for some $r < 0$, then $q_3(S^3_{r'}(K)) = 0$ for all $r' > 0$.

If $M(K) < 0$, then because $q_3(S^3_{-1}(K)) = 1$, we see that $q_3(S^3_r(K)) = 0$ for all $r > 0$. 

If $M(K) \ge 0$, then we know by Proposition \ref{prop:mainthm-Z} that $q_3(S^3_{M(K)+1}(K)) = 0$ and if $M(K) > 0$ we know that $q_3(S^3_{M(K)-1}(K)) = -1$. It follows from monotonicity that $q_3(S^3_r(K)) = -1$ for $r \in (0, M(K)-1)$ and $q_3(S^3_r(K)) = 0$ for $r \in (M(K)+1,\infty)$. To understand the value of $q_3$ for $r \in (M-1,M+1)$ and thereby complete the proof, we must extend this discussion to  compute this invariant for the values of $r$ closest to $M(K)$.

\begin{lemma}
    For any integer $n > 0$ and any knot $K$, we have \[q_3(S^3_{M(K) + 1/n}(K)) = \begin{cases} 1 & M(K) < 0 \\ 0 & M(K) \ge 0 \end{cases}\]
\end{lemma}

A similar description holds in the case of slope $M(K) - 1/n$. By the discussion above, establishing this Lemma completes the proof of Theorem \ref{thm:q3-surgery}.

\begin{proof}
We prove a more technical claim. Decompose
\begin{align*}
I^\#(S^3_M(K),w;\mathbb F) &= V_0 \oplus \gr V_0  \\
I^\#(S^3_{M+1/n}(K);\mathbb F) &= V_n \oplus \mathbb F \oplus \gr V_n.
\end{align*}
We claim that for all $n \ge 1$, there is an isomorphism \[\psi_n: I^\#(S^3_{M+1/n}(K);\mathbb F) \cong V_0^n \oplus \mathbb F \oplus (\gr V_0)^n\] with the following properties:
\begin{itemize} 
\item This isomorphism sends $V_n$ to $V_0^n$, and sends $V_n \oplus \mathbb F$ to $V_0^n \oplus \mathbb F$. That is, the isomorphism preserves the natural $3$-step filtrations.
\item With respect to this isomorphism, the map induced by the $2$-handle cobordism $Z_n: S^3_{M+1/(n+1)}(K) \to S^3_{M+1/n}(K)$ has \[\psi_{n} I^\#(Z_n;\mathbb F) \psi_{n+1}^{-1} = \pi_{n+1},\] where $\pi_{n+1}$ is projection onto the last $n$ coordinates on $V_0^{n+1}$ and $\gr V_0^{n+1}$, and is the identity on the second summand.
\end{itemize}
This will give the value of $q_3$, as in particular $\gr \epsilon_0 = 1$ for each of these cobordisms; by Proposition \ref{prop:q3-eps} this implies $q_3(S^3_{M+1/(n+1)}(K)) = q_3(S^3_{M+1/n}(K))$. Thus, the value for each of these agrees with $q_3(S^3_{M+1}(K))$, which was computed in Proposition \ref{prop:mainthm-Z}. 

The proof of the technical claim is nearly identical to \cite{DLME}*{Lemma 4.2}, which proves a similar statement for the irreducible instanton homology; we provide a proof for completeness. The result follows by induction on $n$, using the distance-$2$ exact triangle \[\cdots \to I^\#(S^3_{M+1/(n+1)}(K);\mathbb F) \xrightarrow{Z_n \oplus Z'_n} I^\#(S^3_{M+1/n}(K);\mathbb F)^2 \xrightarrow{Z'_{n-1} \oplus Z_{n-1}} I^\#(S^3_{M+1/(n-1)};\mathbb F) \to \cdots\] By the dimension formula, \[\dim I^\#(S^3_{M+1/n}(K);\mathbb F)^2 = \dim I^\#(S^3_{M+1/(n+1)}(K);\mathbb F) + \dim I^\#(S^3_{M+1/(n-1)}(K);\mathbb F),\] so the exact triangle reduces to a short exact sequence, which is then isomorphic to \[0 \to I^\#(S^3_{M+1/(n+1)}(K);\mathbb F) \to (V_0^n \oplus \mathbb F \oplus \gr V_0^n)^2 \to V_0^{n-1} \oplus \mathbb F \oplus \gr V_0^{n-1} \to 0\]
where the first map is $\psi_n Z'_n \oplus \psi_n Z_n$ and the second map is $\pi_n \oplus \psi_{n-1} Z'_{n-1} \psi_n^{-1}$. Write \begin{align*}
    \sigma_n&: V_0^{n-1} \oplus \mathbb F \oplus \gr V_0^{n-1} \to V_0^n \oplus \mathbb F \oplus \gr V_0^n \\ 
    \iota_n&: V_0 \oplus \gr V_0 \to V_0^n \oplus \mathbb F \oplus \gr V_0^n
\end{align*}
for inclusion of the last $(n-1)$ and first coordinate, respectively. There is now a filtration-preserving identification \[\varphi_n = \begin{pmatrix} \iota_n & \sigma_n \psi_{n-1} Z'_{n-1} \psi_n^{-1} \\ 0 & 1 \end{pmatrix}: (V_0 \oplus \gr V_0) \oplus (V^n_0 \oplus \mathbb F \oplus \gr V^n_0) \to \ker(\pi_n \oplus \psi_{n-1} Z'_{n-1} \psi_n^{-1}).\]
The map $\psi_{n+1}$ can be defined as the composite $\varphi^{-1}_n \circ (\psi_n Z'_n \oplus \psi_n Z_n)$. Now $\pi_{n+1} \varphi_n^{-1}$ is projection to the second coordinate, so $\pi_{n+1} \psi_{n+1} = \psi_n Z_n$, as desired.
\end{proof}

\section{Torsion-averse knots}\label{sec:TAK}
In this section, we compare the invariants $r_2(K)$ and $M(K)$ to the invariants $r_0(K)$ and $\nu^\#(K)$ studied in \cite{bs-concordance}. Further leveraging the main results of \cite{LiYe1,LiYe2} against our dimension formula, we are able to sharpen and improve their results.

\begin{lemma}\label{lemma:r2r0diff}
We have $r_2(K) \ge r_0(K)$, and the difference satisfies $r_2(K) - r_0(K) \ge |M(K) - \nu^\#(K)|$. 
\end{lemma}
\begin{proof}
Except for possibly two exceptions, we have \[r_2(K) + |n-M(K)| = \dim I^\#(S^3_n(K);\mathbb F) \ge \dim I^\#(S^3_n(K);\mathbb Q) = r_0(K) + |n-\nu^\#(K)|.\] Taking $n$ very large, this simplifies to $r_2(K) - r_0(K) \ge M(K) - \nu^\#(K)$. Taking $n$ very small, this instead simplifies to $r_2(K) - r_0(K) \ge \nu^\#(K) - M(K)$. Combining these gives the stated absolute value inequality.
\end{proof}

We now study the case in which equality is achieved.

\begin{defn}
    A knot $K$ is said to be (positive) \textbf{torsion-averse} if there exists some slope $p/q \in \mathbb Q$ (with $q \ne 0$) and some choice of bundle data $[w]_2 \in H_1(S^3_{p/q}(K);\mathbb F)$ so that $I^\#(S^3_{p/q}(K),w;\mathbb Z)$ has no $2$-torsion. We say $K$ is an \textbf{instanton $L$-space knot over $R$} if $I^\#(S^3_{p/q}(K),w;R) \cong R^{|p|}$ for some slope $p/q$ and some choice of bundle data.
\end{defn}

Equivalently, we ask that $\dim I^\#(S^3_{p/q}(K),w;\mathbb F) = \dim I^\#(S^3_{p/q}(K),w;\mathbb Q)$ for some $p/q$ and $w$. 

If $K$ is an instanton $L$-space knot over $\mathbb F$, then the universal coefficient theorem shows $K$ is torsion-averse and is also an instanton $L$-space knot over $\mathbb Q$. A torsion-averse knot is an instanton $L$-space knot over $\mathbb F$ if and only if it is an $L$-space knot over $\mathbb Q$. 

In fact, the main result of \cite{LiYe2} includes that two of these notions coincide.

\begin{lemma}
$K$ is torsion-averse if and only if it is an instanton $L$-space knot over $\mathbb F$.
\end{lemma}
\begin{proof}
The reverse direction was discussed above. Now suppose $I^\#(S^3_r(K),w;\mathbb Z)$ is torsion-free; without loss of generality, suppose $r \ge 0$. By combining Theorem \ref{thm:rational-dimformula} and Theorem \ref{thm:surgerydim}, we see that if $[w]_2 \ne 0$ we have $r \ne M(K)$, in which the dimension formulas imply that $I^\#(S^3_t(K);\mathbb Z)$ also has no $2$-torsion for any $t > r$.

By \cite{LiYe2}*{Theorem~1.4}, $K$ is an instanton $L$-space knot over $\mathbb Q$ and $t > 2g(K) - 1$. By \cite{bs-concordance}*{Theorem~1.18}, if $t = p/q$, we have $\text{rank } I^\#(S^3_t(K);\mathbb Z) = p$. Because this group has no $2$-torsion, it follows from the universal coefficient theorem that $I^\#(S^3_t(K);\mathbb F) = \mathbb F^p$, so $K$ is also an instanton $L$-space knot over $\mathbb F$.
\end{proof}

Thus if any $I^\#(S^3_r(K),w;\mathbb Z)$ is torsion-free, we have $r_2(K) = |M(K)|$. Applying Theorem \ref{thm:surgerydim} immediately gives the following result.

\begin{corollary}
Suppose $K \subset S^3$ is a knot. Then $I^\#(S^3_r(K),w;\mathbb Z)$ is torsion-free if and only if the following conditions hold: 
\begin{itemize}
    \item $K$ is torsion-averse.
    \item $r$ and $M(K)$ have the same sign.
    \item $|r| \ge |M(K)|$, and if these are equal $[w]_2 \ne 0$. 
\end{itemize}
\end{corollary}

By \cite{bs-concordance}*{Theorem~1.18}, in this case we have $|\nu^\#(K)| = 2g(K) - 1$. Because $|M(K)| \ge |\nu^\#(K)|$ and $M(K)$ is divisible by $4$, we obtain the following result.

\begin{corollary}
If $K \subset S^3$ is any non-trivial knot, the group $I^\#(S^3_r(K);\mathbb Z)$ has $2$-torsion for any $|r| \le 4\lceil g(K)/2\rceil$.
\end{corollary}

The main result of \cite{LiYe2} says more than was discussed above. They consider the quantity \[t_2(K) = \frac 12 \left(\dim I^\#(S^3_1(K);\mathbb F) - \dim I^\#(S^3_1(K);\mathbb Q)\right),\] equal to the number of $2$-torsion summands in $I^\#(S^3_1(K);\mathbb Z)$. They prove that $I^\#(S^3_n(K);\mathbb Z)$ is torsion-free for integers $n$ with $|n| < 2g(K) - 1 + t_2(K)$. We show that their result is nearly sharp.

\begin{prop}
    Suppose $K$ is a torsion-averse knot. Then $|M(K)| = 2g(K) - 1 + t_2(K)$.
\end{prop}
\begin{proof}
Because $K$ is an instanton $L$-space knot over $\mathbb F$ and $\mathbb Q$, we have $r_2(K) = |M(K)|$ and $r_0(K) = |\nu^\#(K)|$. Applying Theorem \ref{thm:rational-dimformula} and Theorem \ref{thm:surgerydim}, we may express \begin{align*}
    2t_2 &= r_2 + |1-M| - r_0 - |1-\nu^\#| \\
    &= |M| - |\nu^\#| + |1-M| - |1-\nu^\#|.
\end{align*}
Notice that $M$ and $\nu^\#$ share the same sign. For instance, if $M$ is positive, then $S^3_{M+1/2}(K)$ is an instanton $L$-space over $\mathbb F$ and hence over $\mathbb Q$. It follows from Theorem \ref{thm:rational-dimformula} that $0 < \nu^\#(K) \le M(K) + 1/2$. Similarly, it follows that $1-M$ and $1-\nu^\#$ share a sign. We may thus combine the above absolute values to give \[2t_2 = |M - \nu^\#| + |M-\nu^\#| = 2|M| - 2|\nu^\#|.\] By \cite{bs-concordance}*{Theorem~1.18}, in this case we have $|\nu^\#(K)| = 2g(K) - 1$. Dividing and rearranging gives the desired expression.
\end{proof}

Many examples of instanton $L$-spaces over $\mathbb F$ arise as follows.

\begin{defn}
A rational homology $3$-sphere $Y$ is called \textbf{nondegenerate $\SU(2)$-abelian} if every homomorphism $\alpha: \pi_1(Y) \to \SU(2)$ has abelian image and $\alpha$ is a nondegenerate critical orbit of the Chern--Simons functional.
\end{defn}

As discussed by Boyer and Nicas \cite{boyer-nicas}, this is equivalent to the following condition: every regular cyclic cover $\tilde Y \to Y$ other than those of maximal even degree has $\tilde Y$ a rational homology sphere. 

More generally, if $w \subset Y$ is a $1$-cycle, the critical points of the Chern--Simons functional correspond to representations $\alpha: \pi_1(Y \setminus w) \to \SU(2)$ up to conjugacy satisfying the following condition: for each component of $w$, a meridian is sent to $-1$. To say that $(Y,w)$ is nondegenerate $\SU(2)$-abelian means that these all have abelian image, in which case $\ad \alpha$ extends to a representation of $\pi_1(Y)$ to $SO(2)$, and nondegeneracy of $\alpha$ is equivalent to asking that the cover corresponding to each $\ker \ad \alpha$ is a rational homology sphere.

\begin{lemma}\label{lemma:ndga-q3}
If $(Y,w)$ is a nondegenerate $\SU(2)$-abelian pair, then $I^\#(Y,w;\mathbb Z) \cong \mathbb Z^{|H_1(Y;\mathbb Z)|}$. If $w = \varnothing$ and $Y$ is a $\mathbb F$-homology sphere, then $q_3(Y) = 0$. 
\end{lemma}
\begin{proof}
    Under these assumptions, we may take the zero perturbation on $(Y,w)$, for which the complex \[\widetilde C(Y,w;\mathbb Z) = A(Y,w) \oplus \chi_2 A(Y,w) \oplus Z(Y,w)\] is generated by reducible flat connections. Each reducible connection lies in even degree, so it follows that $\widetilde C(Y,w;\mathbb Z)$ is concentrated in even degrees, and thus that $\widetilde d = 0$. Using the enumeration of $\mathfrak A(Y,w)$ and $\mathfrak Z(Y,w)$ discussed in Section \ref{subsec:DMES1}, one sees that $\widetilde C(Y,w;\mathbb Z) \cong \mathbb Z^n$ for $n = |H_1(Y;\mathbb Z)|$, concentrated in degree zero.
    The homology $I^\#(Y,w;\mathbb Z)$ is computed as the $\tau$-fixed points in the homology of $(\widetilde C(Y, w; \mathbb Z) \otimes V, \widetilde d + 4\chi_3 \tau)$ where $V = \mathbb Z^2$ and $\tau(x,y) = (y,x)$. However, in this case, $\chi_3$ acts trivially on $\widetilde C(Y,w)$, so the differential is again just zero. Thus, $I^\#(Y,w;\mathbb Z) = \mathbb Z^n$ for $n = |H_1(Y;\mathbb Z)|$, and is concentrated in degree zero.

    Suppose also now that $w = \varnothing$ and $Y$ is a $\mathbb F$-homology sphere. Because the complex $\widetilde C$ has zero differential, in this case, we have $\widehat I_{\chi_2}(Y) \cong \mathbb F_2[x]$ and $\overline I_{\chi_2}(Y) \cong \mathbb F_2\llbracket x^{-1}, x]$. It follows from the definition that $q_3(Y) = 0$.
\end{proof}

\section{Computations}\label{sec:comp}
The purpose of this section is to present a few computations.

\begin{lemma}
Suppose $K'$ is the resulting of modifying the knot $K$ by changing $p$ positive crossings and $n$ negative crossings. Then we have an inequality \[-4n \le M(K) - M(K') \le 4p.\]
\end{lemma}
\begin{proof}
By a mirroring argument it suffices to prove one inequality.

This is a very slight generalization of the argument in \cite{OwensStrle-negdef}*{Theorem 1.1(a)}. If $K'$ is framed with integer $d$, one obtains a diagram for $K$ framed with integer $d+4p$ by placing $(-1)$-framed unknots around each of the changed crossings; compare \cite{GompfStipsicz}*{Figure~5.19}. 

This diagram defines a cobordism $W: S^3_d(K') \to S^3_{d+4p}(K)$ which, if $d$ is odd, has $H_1(W;\mathbb F) = 0$. If $d, d+4p$ have the same sign, then $W$ is negative-definite; otherwise $W$ has $b^+(W) = 1$. 

Suppose first that $M(K') \ge 0$, so $q_3(S^3_{M(K')+1}(K')) = 0$. By monotonicity under negative-definite cobordisms we have $q_3(S^3_{M(K')+4p+1}(K)) \ge 0$, and because this is a positive surgery we must have equality here. Thus $M(K) \le M(K') + 4p$. 

If instead $M(K') < 0$, then $q_3(S^3_{M(K')+1}(K')) = 1$. If $M(K')+4p+1 > 0$, then applying \eqref{eqn:q3-main-ineq} we see that $q_3(S^3_{M(K')+4p+1}(K)) \ge 0$, and because this is a positive surgery we must have equality. Therefore $M(K) \le M(K') + 4p$ in this case as well. 

Finally, if $M(K') + 4p + 1 < 0$, then a similar monotonicity argument gives $q_3(S^3_{M(K')+4p+1}(K)) = 1$, so that $M(K) \le M(K')+4p \le 0$.
\end{proof}

Consider now the family of twist knots, denoted $K_n$ in \cite{bs-concordance}. We have $K_1 = 3_1$ equal to the left-handed trefoil, while $K_2 = 4_1$ is the figure-eight knot. 

\begin{theorem}
For the twist knots $K_n$, the invariants $r_2$ and $M$ are \begin{align*}
    M(K_{2m-1}) = -4, \quad & \quad r_2(K_{2m-1}) = 8m-4 \\
    M(K_{2m}) = 0 \quad & \quad r_2(K_{2m}) = 8m.
\end{align*}
\end{theorem}
\begin{proof}
For $K_1$ these were computed in Example \ref{ex:trefoil}. The knot $K_2 = 4_1$ is amphicheiral, so $M(K_2) = 0$. 

Observe that each $K_{2m-1}$ can be unknotted by changing one negative crossing, so $-4 \le M(K_{2m-1})$. In addition, for $m \ge 1$ each $K_{2m+1}$ can be modified to $K_{2m-1}$ by changing one negative crossing, so $-4 \le M(K_{2m+1}) - M(K_{2m-1}) \le 0$, and in particular $-4 \le M(K_{2m-1}) \le M(K_1) = -4$ for all $m \ge 1$. 

Each $K_{2m}$ can be unknotted by changing one positive crossing, so $0 \le K_{2m} \le 4$. On the other hand, it is still the case that $-4 \le M(K_{2m+2}) - M(K_{2m}) \le 0$ for all $m \ge 1$, so that $0 \le M(K_{2m}) \le M(K_2) = 0$ for all $m \ge 1$. 

Now as discussed in \cite{bs-concordance}*{Proposition~7.4-7.5}, there exist identifications \[S^3_1(K_{2m}) \cong S^3_{-1/m}(3^*_1) \cong \Sigma(2,3,6m+1), \quad S^3_{-1}(K_{2m-1}) \cong S^3_{-1/m}(3_1) \cong \Sigma(2,3,6m-1).\]
Theorem \ref{thm:Brieskorn-calc} asserts in particular that
\[\dim I^\#(\Sigma(2,3,6m+1);\mathbb F) = 8m+1, \quad \dim I^\#(\Sigma(2,3,6m-1);\mathbb F) = 8m-1.\] Then Theorem \ref{thm:surgerydim} gives $r_2(K_{2m}) = 8m$ and $r_2(K_{2m-1}) = 8m-4$. 
\end{proof}

We may also argue that there exists an infinite family of knots with $r_2(K) + |M(K)|$ bounded. 

\begin{theorem}
Consider the family of pretzel knots $P_n = P(n,-3,3)$. For all $n \in \mathbb Z$, we have $r_2(P_n) = 16$ and $M(P_n) = 0$.
\end{theorem}
\begin{proof}
The argument is akin to the proof of \cite{bs-concordance}*{Theorem~1.15}, and we use facts stated in that proof without further justification. Because $P_n$ is slice, we have $M(P_n) = 0$. Writing $r_2(n) = r_2(P(n,-3,3))$, because \[S^3_{-2}(P(n,-3,3)) \cong S^3_2(P(n+3,-3,3)), \quad S^3_1(P(3,3,-3)) \cong -S^3_1(P(4,3,-3)).\] The first gives that $r_2(n)$ depends only on $n$ mod $3$, while the second gives $r_2(3) = r_2(4)$. The knots $P(\pm 1, -3, 3)$ are $K_4$ and its mirror, which has $r_2(K_4) = 16$, so $r_2(1) = r_2(-1) = 16$. It follows that $r_2(n) = 16$ for all $n$.
\end{proof}

\bibliography{biblio.bib}
\bibliographystyle{alpha}
\end{document}